\documentclass{psp}

\usepackage{amsfonts,amssymb,amsmath,euscript}
\usepackage[colorlinks,citecolor=blue,linkcolor=red]{hyperref}
\usepackage{color}

\usepackage{graphicx,setspace,amscd,float,hyperref,enumitem}

\newtheorem{thm}{Theorem}[section]
\newtheorem{lem}[thm]{Lemma}

\newtheorem{prop}[thm]{Proposition}
\newtheorem{cor}[thm]{Corollary}
\newtheorem{df}[thm]{Definition}
\newtheorem{rk}[thm]{Remark}

\newtheorem{conv}[thm]{Convention}
\newtheorem{propdef}[thm]{Proposition-Definition}

\newtheorem{notation}[thm]{Notation}

\newcommand{\N}{\mathbb N}

\def\epsilon{\varepsilon}
\def\phi{\varphi}

\def\hat{\widehat}

\newcommand{\Aut}{\mbox{Aut}}

\newcommand{\fp}{f_{prim}}
\newcommand{\fs}{f_{simp}}
\newcommand{\ff}{f_{fill}}
\newcommand{\dpr}{d_{prim}}
\newcommand{\dsi}{d_{simp}}
\newcommand{\dfi}{d_{fill}}

\newcommand{\cvn}{\mbox{cv}_N}
\newcommand{\cvnbar}{\overline{\mbox{cv}}_N}
\newcommand{\CVN}{\mbox{CV}_N}
\newcommand{\CVNbar}{\overline{\mbox{CV}}_N}

\DeclareMathOperator{\RF}{RF}
\DeclareMathOperator{\D}{D}

\usepackage{lipsum}

\newcommand\blfootnote[1]{%
  \begingroup
  \renewcommand\thefootnote{}\footnote{#1}%
  \addtocounter{footnote}{-1}%
  \endgroup
}

\begin{document}

\title[The primitivity index function for a free group]{The primitivity
  index function for a free group, and untangling closed curves on hyperbolic surfaces.\\ \emph{With an appendix by Khalid Bou-Rabee.}}

\author[Neha Gupta and Ilya Kapovich]{NEHA GUPTA\\ Department of Mathematics, FAS, Harvard University, One Oxford Street, Cambridge, MA 02138  \addressbreak
 e-mail \texttt{neha@math.harvard.edu}
 \and\ 
 ILYA KAPOVICH\\ Department of Mathematics, University of Illinois at Urbana-Champaign,\addressbreak
  1409 West Green Street, Urbana, IL 61801\addressbreak 
   email \texttt{kapovich@math.uiuc.edu}}

\maketitle  

\begin{abstract}
Motivated by the results of Scott and Patel about ``untangling" closed geodesics in finite covers of hyperbolic surfaces, we introduce and study primitivity, simplicity and non-filling index functions for finitely generated free groups. We obtain lower bounds for these functions and relate these free group results back to the setting of hyperbolic surfaces.
An appendix by Khalid Bou-Rabee connects the primitivity index function $\fp(n,F_N)$  to the residual finiteness growth function for $F_N$.
\end{abstract}

\thanks{The second author was partially supported by the NSF grants DMS-1405146 and DMS-1710868.  Both authors acknowledge support from U.S. National Science Foundation grants DMS 1107452, 1107263, 1107367 ``GEAR Network".}

%\subjclass[2010]{Primary 20F65, Secondary  37D, 53C, 57M}

%\date{}

\section{Introduction}

\blfootnote{2010 \emph{Mathematics Classification.}  Primary 20F65, Secondary  37D, 53C, 57M.}

Let $\Sigma$ be a compact connected surface  with a hyperbolic metric $\rho$ and with (possibly empty) geodesic boundary. In \cite{Sc1,Sc2} Scott proved that $\pi_1(\Sigma)$ is \emph{subgroup separable} or \emph{LERF}, meaning that for every finitely generated subgroup $K\le \pi_1(\Sigma)$ and every $g\in \pi_1(\Sigma)$ such that $g\not\in K$ there exists a subgroup $H\le \pi_1(\Sigma)$ of finite index in $\pi_1(\Sigma)$ such that $K\le H$ but $g\not\in H$. (Scott's result dealt with the case of a closed surface $S$ since in the case $\partial S\ne \emptyset$, the group $\pi_1(S)$ is free and hence known to be subgroup separable by a much older result of Hall~\cite{MH}).
In the same work~\cite{Sc1} Scott showed that  if $\gamma$ is a closed geodesic on $(\Sigma,\rho)$  then there exists a finite cover $\hat\Sigma\to\Sigma$ such that $\gamma$ lifts to a simple closed geodesic in $\hat \Sigma$, where $\hat\Sigma$ is given the hyperbolic structure obtained by the pull-back of $\rho$. As customary in the context of hyperbolic surfaces, the term ``closed geodesic" here assumes that the curve in question is not a proper power in the fundamental group of the surface.
Recently Patel~\cite{Patel} obtained quantitative versions of Scott's subgroup separability result and of his result about lifting a closed geodesic to a simple one in a finite cover.
Thus she proved that for every $\Sigma$ as above there exists a hyperbolic metric $\rho_0$ on $\Sigma$ such that every closed geodesic of length $L$ on $(\Sigma,\rho_0)$ lifts to a simple closed geodesic in some finite cover of $\Sigma$ of degree $\le 16.2L$.  Since the length functions on $\pi_1(\Sigma)$ coming from any two hyperbolic structures on $\Sigma$ are bi-Lipschitz equivalent, it follows that for any hyperbolic structure $\rho$ on $\Sigma$ there is some constant $c>0$ such that every closed geodesic of length $L$ on $(\Sigma,\rho)$ lifts to a simple closed geodesic in some finite cover of $\Sigma$ of degree $\le cL$. 
Motivated by these results, if $\rho$ is a hyperbolic structure on $\Sigma$, for every closed geodesic $\gamma$ on $(\Sigma,\rho)$ we denote by $\deg_{\Sigma,\rho}(\gamma)$ the smallest degree of a finite cover of $\Sigma$ such that $\gamma$ lifts to a simple closed geodesic in that cover. For $L\ge sys(\rho)$ (where $sys(\rho)$ is the shortest length of a closed geodesic on $(\Sigma,\rho)$) put $f_{\Sigma,\rho}(L)$ to be the maximum of $\deg_{\Sigma,\rho}(\gamma)$ taken over all closed geodesics $\gamma$ on $(\Sigma,\rho)$ of length $\le L$. Patel's result mentioned above implies that for every hyperbolic  structure $\rho$ on $\Sigma$ there is $c>0$ such that $f_{\Sigma,\rho}(L)\le cL$ for all $L\ge sys(\rho)$.

%\end{document}

A simple closed geodesic on a hyperbolic surface is a particular example of a non-filling curve. Thus for a hyperbolic surface $(\Sigma,\rho)$ as above and for a closed geodesic $\gamma$ on $\Sigma$ we can also define $\deg_{\Sigma,\rho}^{fill}(\gamma)$ to be the smallest degree of a finite cover of $\Sigma$ such that $\gamma$ lifts to a non-filling closed geodesic in that cover. Then put $f_{\Sigma,\rho}^{fill}(L)$ to be the maximum of $\deg_{\Sigma,\rho}^{fill}(\gamma)$ taken over all closed geodesics $\gamma$ on $(\Sigma,\rho)$ of length $\le L$. Thus, in view of Patel's result, we have $f_{\Sigma,\rho}^{fill}(L)\le f_{\Sigma,\rho}(L)\le cL$ for all $L\ge sys(\rho)$.
However, up to now, nothing has been known about lower bounds for $f_{\Sigma,\rho}(L)$ or $f_{\Sigma,\rho}^{fill}(L)$. (Note that the first place where the question about quantitative properties of $f_{\Sigma,\rho}(L)$ was raised, although somewhat indirectly, appears to have been the paper of Rivin~\cite{Rivin}). 

In general, obtaining lower bounds for quantitative results related to residual finiteness is quite difficult, and is usually harder than obtaining upper bounds.
Recently there has been a significant amount of research regarding quantitative aspects of residual finiteness; see, for example~\cite{BR1,BBRKM,BRK,BRMR1,BRMR2, BRMR3,BRHP,BRS,GV,KaMa,KoTh,Patel,Rivin,Buskin,BRCor,BRMy,BRSt}.  We will discuss some of these results in more detail below.

Let $N\ge 2$ be an integer and let $F_N$ be the free group of rank $N$.  If $A$ is a free basis of $F_N$, for an element $g\in F_N$ we denote by $|g|_A$ the freely reduced length of $g$ over $A$ and we denote by $||g||_A$ the cyclically reduced length of $g$ over $A$.
A classic result of Marshall Hall~\cite{MH}, mentioned above, (see also \cite{KM} for a modern proof using Stallings subgroup graphs) proves that finitely generated free groups are subgroup separable. More precisely, Hall proved that if $K\le F_N$ is a finitely generated subgroup and $g\in F_N-K$ then there exists a subgroup $H\le F_N$ of finite index such that $g\not\in H$, $K\le H$, and, moreover, $K$ is a free factor of $H$.  It is not hard to adapt the proof of this result to show that for every $g\in F_N, g\ne 1$ there exists a subgroup $H\le F_N$ of finite index such that $g\in H$ and that $g$ is a \emph{primitive} element of $H$, that is, that $g$ belongs to some free basis of $H$. In fact, a simple argument using Stallings subgroup graphs (see Proposition~\ref{p:comp}  below) shows that if $A$ is a free basis of $F_N$ and $w$ is a nontrivial cyclically reduced word in $F(A)$ of length $n$  then there exists a subgroup $H\le F_N$ with $[F_N:H]=n$ such that $w\in H$ is a primitive element of $H$.
For a nontrivial element $g\in F_N$ we define the \emph{primitivity index} $\dpr(g)=\dpr(g;F_N)$ as the minimum of $[F_N:H]$ where $H$ varies over all subgroups of finite index in $F_N$ containing $g$ as a primitive element. Given a free basis $A$ of $F_N$, for $n\ge 1$ we then define $\fp(n)=\fp(n;F_N)$ as the maximum of $\dpr(g)$ where $g$ varies over all nontrivial freely reduced words of length $\le n$ in $F_N=F(A)$ which are not proper powers in $F_N$. It is not hard to see that $\fp(n)$ does not depend on the choice of a free basis $A$ of $F_N$; we call $\fp(n)$ the \emph{primitivity index function} for $F_N$.   
Thus $\fp(n)$ is the smallest monotone non-decreasing function such that for every nontrivial root-free $g\in F_N$ we have $\dpr(g)\le \fp(|g|_A)$. 

A nontrivial element $g\in F_N$ is called \emph{simple} in $F_N$ if $g$ belongs to some proper free factor of $F_N$. A nontrivial element $g\in F_N$ is called \emph{filling} in $F_N$ if $g$ does not belong to a vertex group of a nontrivial splitting of $F_N$ over the trivial or maximal infinite cyclic subgroup. See Section~\ref{s:elements} for more precise definitions and a discussion of these notions. Note that for $1\ne g\in F_N$, if $g$ is primitive then $g$ is simple, and if $g$ is simple then $g$ is non-filling.  For a nontrivial element $g\in F_N$ let $\dsi(g)=\dsi(g;F_N)$
be the smallest index $[F_N:H]$ where $H$ varies over all subgroups of
finite index in $F_N$ such that $g\in H$ and that $g$ is simple in $H$. Finally,  let $\dfi(g)=\dfi(g;F_N)$
be the smallest index $[F_N:H]$ where $H$ varies over all subgroups of
finite index in $F_N$ such that $g\in H$ and that $g$ is non-filling in $H$.  Thus by  definition, we have
$\dfi(g)\le \dsi(g)\le \dpr(g)$. For $n\ge 1$ we then define the \emph{simplicity index function} $\fs(n)=\fs(n;F_N)$ as the maximum
of $\dsi(g)$ where $g$ varies over all nontrivial freely reduced words
of length $\le n$ in $F_N=F(A)$ that are not proper powers in $F_N$. Also, for $n\ge 1$ we then define the \emph{non-filling index function} $\ff(n)=\ff(n;F_N)$ as the maximum
of $\dfi(g)$ where $g$ varies over all nontrivial freely reduced words
of length $\le n$ in $F_N=F(A)$ that are not proper powers in $F_N$.

In view of Proposition~\ref{p:comp}
mentioned above, for every nontrivial $g\in F_N$ we have $\dsi(g)\le
\dpr(g)\le ||g||_A\le |g|_A$, and hence $\ff(n)\le \fs(n)\le \fp(n)\le n$ (see
Lemma~\ref{l:aut} for details).

In general, we are interested in the following types of questions:

\begin{itemize}
\item Understand the actual asymptotics of the ``worst-case" index functions $\ff(n)$, $\fs(n)$,  $\fp(n)$ for free groups and of their geometric counterparts $f_{\Sigma,\rho}(L)$ or $f_{\Sigma,\rho}^{fill}(L)$.
\item Find specific sequences of elements in free groups or curves on surfaces realizing this ``worst-case" behavior or at least exhibiting reasonably fast growth of the corresponding index and degree functions.
\item Understand the asymptotics of the indexes $\dpr(g_n),\dsi(g_n),\dfi(g_n)$ and of $\deg_{\Sigma,\rho}(\gamma_n)$, $\deg_{\Sigma,\rho}^{fill}(\gamma_n)$ for various ``natural" sequences of group elements $g_n\in F_N$ or closed geodesics $\gamma_n$ on $(\Sigma,\rho)$.
\item Understand the relationship between the index functions for free groups and the degree functions for surfaces, and relate both to other functions measuring quantitative aspects of residual properties of free and surface groups.
\end{itemize}

Our first main result provides a lower bound for $\ff(n; F_N)$; see Theorem~\ref{t:newmain} below:

\begin{thm} \label{thm:A}
Let $N\geq2$ and let $F_N = F(A)$ where $A = a_1, \ldots, a_N$. Then there exists a constant $c > 0$ and an integer $M\ge 1$ such that 
for all $n \geq M$ we have
\[
\fp(n) \geq \fs(n) \geq \ff(n) \ge c \frac{\log n}{ \log \log n}.
\]
\end{thm}

For a finitely generated group $G$ equipped with a finite generating set $A$, the \emph{residual finiteness growth function} $\RF_{G}(n)$ is defined as the smallest number $d$ such that for every nontrivial element $g\in G$ of word-length $\le n$ with respect to $A$ there exists a subgroup of index at most $d$ in $G$ that does not contain $g$.

In an appendix to this paper, for a free group $F_N$ with a free basis $A$,  Khalid Bou-Rabee relates $\fp(n,F_N)$ to the residual finiteness growth function $\RF_{F_N}(n)$. Namely, he shows in Theorem~\ref{t:a1} below
that for $n\ge 1$ one has $\fp(4n+4,F_N)\ge  \RF_{F_N}(n)$.  Using a recent result of Kozma and Thom~\cite{KoTh} about lower bounds for $\RF_{F_N}(n)$, Bou-Rabee then shows in Corollary~\ref{c:a2} below that for all sufficiently large $n$ one has \[
\fp(4n+4) \geq \exp \left( \left( \frac{\log(n)}{C \log\log(n)} \right)^{1/4} \right).
\]
Note that this lower bound behaves almost like $n^{1/4}$.  Moreover, if we assume Babai's Conjecture on the diameter of Cayley graphs of permutation groups, then for all sufficiently large $n$ we have an almost linear lower bound:
\[
\fp(4n+4) \geq n^{\frac{1}{C\log\log(n)}}.
\]
Bou-Rabee's homological trick used in Theorem~\ref{t:a1}  does not work for the index functions $\fs(n)$ and $\ff(n)$. Thus for these functions the lower bound given by Theorem~\ref{thm:A} remains the best known bound. 

We also obtain a bound from below on $\dsi(w_n)$ and $\dfi(w_n)$ where $w_n$ is a ``random" freely reduced word in $F(A)$ of length $n>>1$.

\begin{thm}\label{thm:B}
Let $N\ge 2$ and let $F_N=F(A)$ where $A=\{a_1,\dots, a_N\}$.  

Then there exist  constants $c(N)>0$, $D_1(N)>1$, $1>D_2(N)>0$  such that for $n\ge 1$ and for a freely reduced word $w_n\in F(A)$ of length $n$ chosen uniformly at random from the sphere $S(n)$ of radius $n$ in $F(A)$ we have
\[
1-P_{\mu_n}\left(\dsi(w_n)\ge c\log^{1/3} n\right)  = O\left((D_1)^{-n^{D_2}}\right)
\]
and 
\[
1-P_{\mu_n}\left(\dfi(w_n)\ge c\log^{1/5} n\right)  = O\left((D_1)^{-n^{D_2}}\right)
\]
so that
\[
\lim_{n\to\infty} P_{\mu_n}\left(\dsi(w_n)\ge c\log^{1/3} n\right) =1
\]
and 
\[
\lim_{n\to\infty} P_{\mu_n}\left(\dfi(w_n)\ge c\log^{1/5} n\right) =1
\]
\end{thm}

Here $\mu_n$ is the uniform probability distribution on the $n$-sphere $S(n)\subseteq F_N=F(A)$. See Convention~\ref{conv:O} for our use of the big-O notation.

It remains an interesting question to understand the actual behavior of $\dsi(w_n)$ and $\dfi(w_n)$ on ``random" elements $w_n\in F_N$ and, in particular, to see if $\dsi(w_n)$ and $\dfi(w_n)$ admit sublinear upper bounds.

Finally, in Section~\ref{s:surfaces} we relate the above results for free groups to the original motivating questions about the degree functions for hyperbolic surfaces.
Thus, using Theorem~\ref{thm:B}, we obtain (see Theorem~\ref{t:sigmafill} below):

\begin{cor}\label{cor:C}
Let $(\Sigma,\rho)$ be a compact connected hyperbolic surface with $b\ge 1$ geodesic boundary components. Then there exists $C'>0$ such that for all sufficiently large $L$ we have
\[
f_{\Sigma,\rho}(L)\ge f_{\Sigma,\rho}^{fill}(L)\ge C' \frac{\log L}{\log\log L}.
\]
\end{cor}

Similarly, using Theorem~\ref{thm:A}, we obtain (see Theorem~\ref{t:main'} below):

\begin{cor}\label{cor:D}
Let $\Sigma$ be a compact  connected surface  with a hyperbolic structure $\rho$ and with (possibly empty) geodesic boundary.
Let $\Sigma_1\subseteq \Sigma$ be a compact connected subsurface with  $\ge 3$ boundary components, each of which is a geodesic in $(\Sigma,\rho)$. 
Let $x\in \Sigma_1$ and let $A$ be a free basis of $\pi_1(\Sigma_1,x)$.

Let $w_n\in F(A)=\pi_1(\Sigma_1,x)$ be a freely reduced word of length $n$ over $A^{\pm 1}$ generated by a simple non-backtracking random walk on $F(A)=\pi_1(\Sigma_1,x)$.
Let $\gamma_n$ be the closed geodesic on $(\Sigma,\rho)$ in the free homotopy class of $w_n$.

Then there exist  constants $c>0, K'\ge 1$ such that

\[
\lim_{n\to\infty} Pr( \deg_{\Sigma,\rho}(\gamma_n) \ge c \log^{1/3} n) =1
\]
and such that with probability tending to $1$ as $n\to\infty$
we have that $w_n\in \pi_1(\Sigma,x)$ is not a proper power and that $n/K'\le \ell_\rho(\gamma_n)\le K'n$.
\end{cor}

In the original November 2014 version of this paper we used Corollary~\ref{cor:D} to obtain, for all sufficiently large $L$, a lower bound 
\[
f_{\Sigma,\rho}(L)\ge c \log^{1/3} L,
\]
where $(\Sigma,\rho)$ is a closed hyperbolic surface. At the time this was the only known lower bound for $f_{\Sigma,\rho}(L)$. Motivated by our work, Jonah Gaster~\cite{Gas} subsequently obtained a linear lower bound $f_{\Sigma,\rho}(L)\ge cL$ and exhibited a specific sequence of curves $\gamma_n$ in $\Sigma$, living in a pair-of-pants subsurface of $\Sigma$, realizing this lower bound. Since these curves are already non-filling in $\Sigma$ and have $\deg_{\Sigma,\rho}^{fill}(\gamma_n)=1$, Gaster's proof does not provide any lower bounds for $f_{\Sigma,\rho}^{fill}(L)$. Thus for the moment the lower bound for $f_{\Sigma,\rho}^{fill}(L)$ given by Corollary~\ref{cor:C} remains the best bound known.
In Section~\ref{s:surfaces} we also relate our results to the versions of $f_{\Sigma,\rho}(L)$ and $f_{\Sigma,\rho}^{fill}(L)$ that do not involve a hyperbolic metric and use the geometric intersection number $i([\gamma],[\gamma])$ instead of the hyperbolic length of $\gamma$ in their definitions.

Also, in Section~\ref{sect:alg} we prove algorithmic computability of the indexes $\dpr(g,F_N)$ $\dsi(g,F_N)$, $\dfi(g,F_N)$ and of the corresponding index functions $\fp(n), \fs(n), \ff(n)$; see Theorem~\ref{t:algorithm}  and Theorem~\ref{t:algorithm1}  below.

A recent paper of Puder~\cite{Puder} (see also \cite{PP,PuderWu} for related work) introduces the notion of a \emph{primitivity rank} $\pi(g)$ for an element $g\in F_N$. Namely, $\pi(g)$ is defined as the smallest rank of a subgroup $H\le F_N$ such that $g\in H$ but $g$ is not primitive in $H$. Puder proves in \cite[Corollary 4.2]{Puder} that for an element $g\in F_N$ one has either $\pi(g)=\infty$ or $0\le \pi(g)\le N$, and that every integer between $0$ and $N$ does occur as a value of $\pi(g)$ for some $g$. He also defines and studies the primitivity rank $\pi(H)$ for a finitely generated subgroup $H\le F_N$, where $\pi(H)$ is defined as the minimum rank of $J$ such that $H\le J\le F_N$ and that $H$ is not a proper free factor of $J$. These notions are related to and in some sense dual to our definitions of $\dpr(g)$ and $\dsi(g)$, but the precise connection of our results with Puder's work remains to be understood.  Malestein and Putman~\cite{MP} obtained a number of lower bound results (in terms of $k$) for the minimal self-intersection number of  nontrivial elements of the $k$-term of the lower central series and the derived series of a surface group. It would be interesting to see if their methods can be used to obtain lower bounds for the function $f_{\Sigma,\rho}$. It would also be interesting to investigate if looking inside the lower central series and the derived series of $F_N$ may produce new lower bounds for $\fp(n)$ and $\fs(n)$. 

We are grateful to Yuliy Baryshnikov for providing us with a proof of Lemma~\ref{lem:Barysh}. We then used the idea of the proof of Lemma~\ref{lem:Barysh} to obtain Proposition~\ref{p:v}, which plays a crucial role in the proof of our main results. We are also grateful to Igor Rivin for suggesting to try to apply our free group results to untangling closed curves on hyperbolic surfaces, and to Priyam Patel for suggesting to apply our results to the degree functions based in the self-intersection number rather than the length of a curve. We thank Kasra Rafi for the suggestion to consider $\deg_{\Sigma,\rho}^{fill}$ and $f_{\Sigma,\rho}^{fill}$.  We thank Nathan Dunfield and Chris Leininger for many useful conversations.
We are grateful to Andreas Thom,  Gady Kozma,  Doron Puder and Khalid Bou-Rabee for helpful feedback.  We are particularly grateful to the referee of the original version of this paper for pointing out that our methods implied a much better lower bound for $\fp(n)$ and $\fs(n)$ than the one we originally had in mind. We are also grateful to the referee of the current version for numerous detailed helpful suggestions.

\section{Preliminaries}

\subsection{Graphs and Edge Paths}

The exposition in this subsection follows that of~\cite{KP}. 

\begin{df}\label{d:Graph}
A \textit{graph} is a 1-dimensional cell-complex. The 0-cells of $\Gamma$ are called \textit{vertices} and we denote the set of vertices of $\Gamma$ by $V\Gamma$. The open 1-cells of $\Gamma$ are called \textit{topological edges} of $\Gamma$ and the set of topological edges are denoted by $E_{top}\Gamma$. 
\end{df}
Every topological edge of $\Gamma$  is homeomorphic to the open interval $(0, 1)$ and thus, when viewed as a 1-manifold,
admits two possible orientations. An \textit{oriented edge} of $\Gamma$ is a topological edge with a choice of orientation
on it. We denote by $E\Gamma$ the set of all oriented edges of $\Gamma$. If $e \in E\Gamma$ is an oriented edge, we denote by $\bar{e}$ the
same underlying edge with the opposite orientation. Note that for every $e \in E\Gamma$ we have $\bar e \neq e$ and $\bar{\bar{e}} = e$;
thus $\ \bar { }: E\Gamma \rightarrow E\Gamma$ is an involution with no fixed points.

Since $\Gamma$ is a cell-complex, every oriented edge $e \in E\Gamma$ comes equipped with the orientation-preserving
attaching map $j_e : [0, 1] \rightarrow \Gamma$ such that $j_e$ maps $(0, 1)$ homeomorphically to $e$ and such that $j_e(0)$, $j_e(1) \in V\Gamma$ (though not necessarily distinct). For $e \in E\Gamma$ we call $j_e(0)$ the \textit{initial vertex} of e, denoted $o(e)$, and we call $j_e(1)$ the
\textit{terminal vertex} of e, denoted $t(e)$. Thus, by definition, $o(\bar{e}) = t(e)$ and $t(\bar{e}) = o(e)$.

For any vertex $x \in V\Gamma$, the \textit{degree of $x$} in $\Gamma$ denoted by $deg(x)$ is the cardinality of the set $\{e \in E\Gamma | o(e)= x \}$. 

An \emph{orientation} on a graph $\Gamma$ is a partition $E\Gamma = E_+\Gamma \sqcup E_-\Gamma$ such that for an edge $e \in E\Gamma$ we have $e\in E_+\Gamma$ if and only if $\bar{e} \in E_-\Gamma$.
If $\Gamma$ is a graph with an orientation, and $\Delta\subseteq \Gamma$ is a subgraph, then $\Delta$ inherits an \emph{induced orientation} from $\Gamma$ by setting $E_+\Delta:=E_+\Gamma\cap E\Delta$ and $E_-\Delta:=E_-\Gamma\cap E\Delta$. Whenever we are dealing with a graph, equipped with an orientation, and a subgraph of that graph, we will always assume that the subgraph is given the induced orientation. 

An \emph{edge-path} $p$ in $\Gamma$ is a sequence of edges $e_1, e_2, \ldots, e_k$ with $e_i \in E\Gamma$ for all $i$ and $o(e_j)=t(e_{j-1})$ for all $2 \leq j \leq k$. The length $|p|$, of the path $p$ is the number of edges in $p$, that is $|p|=k$. We put $o(p)=o(e_1)$, and $t(p)= t(e_k)$. We define $p^{-1}:= e_k, e_{k-1}, \ldots, e_1$. 
A path $p$ in a graph $\Gamma$ is \emph{reduced} if it does not contain any sub-paths of the form $e,e^{-1}$ where $e \in E\Gamma$ is an edge. 

Note that if $\Gamma$ is a graph and $x\in V\Gamma$ is a vertex, there is a canonical identification of $\pi_1(\Gamma,x)$ with the set of reduced edge-paths from $x$ to $x$ in $\Gamma$. We will use this identification throughout the paper.

\begin{df}\label{d:Graph-map}
For two graphs $\Gamma_1$ and $\Gamma_2$, a morphism or a graph-map $f : \Gamma_1 \rightarrow \Gamma_2$ is a continuous
map $f$ such that $f(V\Gamma_1) \subseteq V\Gamma_2$ and such that the restriction of $f$ to any topological edge $e \in \Gamma_1$ is a homeomorphism between $e$ and some topological edge $e'$ of $\Gamma_2$. Thus a morphism $f : \Gamma_1 \rightarrow \Gamma_2$ naturally defines functions $f : E\Gamma_1 \rightarrow E\Gamma_2$ and $f: V\Gamma_1 \rightarrow V\Gamma_2$ such that for any $e \in E\Gamma_1$ we have $f(\bar{e}) = \overline{f(e)} \in E\Gamma_2$, $o(f(e))= f(o(e))$ and $t(f(e))= f(t(e))$.
\end{df}

\begin{df}\label{d:Core-Graph}
Let $\Gamma$ be a graph and $x \in V\Gamma$. Then the core of $\Gamma$ at $x$ is defined as:
$$Core(\Gamma, x) = \cup\{p\, |\,  \text{where} \, p \, \text{is a reduced path in} \, \Gamma \, \text{from} \, x \, \text{to} \, x \}.$$

Note that $Core(\Gamma, x)$ is a connected subgraph of $\Gamma$ containing $x$. If $Core(\Gamma, x) = \Gamma$ we say
that $\Gamma$ is a \emph{core graph with respect to $x$}.  The graph $Core(\Gamma, x)$ has no degree 1 vertices except possibly $x$ itself. 

We say that a graph $\Gamma$ is a \emph{core graph}  if $\Gamma$ is connected and for every vertex $x\in V\Gamma$ we have $Core(\Gamma, x) = \Gamma$.
\end{df}

If a graph $T$ is a tree then for vertices $v,v'\in VT$ we denote by $[v,v']_T$ the unique reduced edge-path from $v$ to $v'$ in $T$.

\begin{propdef} \label{d:dual}
Let $\Gamma$ be a connected graph, and $x \in V\Gamma$. 
Choose a maximal subtree $T \subseteq \Gamma$, and an orientation $E\Gamma= E_+\Gamma\sqcup E_-\Gamma$. For $e \in E\Gamma$ define $[x, o(e)]_T$ to be the unique reduced path in $T$ from $x$ to $o(e)$, and let $s_e := [x, o(e)]_T \, e \,  [t(e),x]_T$. Let $S_T := \{s_e \, | \, e \in E_+\Gamma - E_+T \}$. 
Then $\pi_1(\Gamma, x)$ is free and $S_T$ is a free basis of $\pi_1(\Gamma, x)$. 

We call $S_T$ the free basis of $\pi_1(\Gamma, x)$ \emph{dual to $T$}.
\end{propdef}

We need to explicitly say how to rewrite elements of $\pi_1(\Gamma,x)$ in terms of the basis $S_T$, both as freely reduced words and cyclically reduced words.

\begin{prop}\label{p:rw-basis} 
Let $\Gamma$ be a connected graph, let $x \in V\Gamma$ and let $T \subseteq \Gamma$ be a maximal subtree.
Suppose $E_+\Gamma - E_+T = \{e_1, \ldots, e_m \}$ where $e_i\ne e_j$ for $i\ne j$, so that $S_T= \{s_{e_i} | 1 \leq i \leq m \}$. Then:
\begin{enumerate}
\item \textit{Rewriting $\gamma$ as a freely reduced word in $S_T$}: Delete from $\gamma$ all edges of $T$ and replace each $e_i^{\pm 1}$ by $s_{e_i}^{\pm 1}$. The result is a freely reduced word over $S_T$ representing $\gamma\in\pi_1(\Gamma,x)$.
\item \textit{Rewriting $\gamma$ as a cyclically reduced word in $S_T$}: First cyclically reduce the edge-path $\gamma$ by removing  the maximal initial and terminal segments of $\gamma$ that cancel in the concatenation $\gamma\gamma$. The result is a subpath $\gamma_1$ of $\gamma$ such that $\gamma_1$ is a closed cyclically reduced path (though $\gamma_1$ maybe based at a vertex different from $x$). Now apply the previous procedure to $\gamma_1$: delete all edges of $T$ and replace each $e_i^{\pm 1}$ by $s_{e_i}^{\pm 1}$. The result is the cyclically reduced form of $\gamma\in \pi_1(\Gamma,x)$ over $S_T$. 
\end{enumerate}

\end{prop}

\subsection{Graphs and subgroups}

In a seminal paper from 1983 Stallings~\cite{St1} used labeled graphs to study subgroups of free groups. We give a brief exposition of the relevant definitions and results below and refer the reader to\cite{KM} for details.

Recall that we fix for the free group $F_N=F(A)=F(a_1,\dots, a_N)$ (where $N\ge 2$), a distinguished free basis $A=\{a_1,\dots, a_N\}$. If $w$ is a word in $\Upsilon= A \sqcup A^{-1}$, we will denote by $\underline{w}$ the freely reduced word in $\Upsilon$ obtained from $w$ by performing all possible (if any) free reductions. 

\begin{df}\label{d:A-graph}
An $A$-graph $\Gamma$ consists of an underlying oriented graph where every edge $e \in E\Gamma$ is labeled by a letter $\mu(e) \in  A\sqcup A^{-1}$ in such a way that $\mu(\bar{e})= (\mu(e))^{-1}$. Multiple edges between vertices and loops at a vertex are allowed. An $A$-graph $\Gamma$ is said to be \textit{folded} if there does not exist a vertex $x$ and two distinct edges $e_1$, $e_2$ with $o(e_1)= o(e_2)= x$ such that $\mu(e_1) = \mu(e_2)$. Otherwise $\Gamma$ is said to be \textit{non-folded}.
\end{df}

An $A$-graph $\Gamma$ is said to be \emph{$A$-regular} if for every vertex $x \in V\Gamma$ and for every $a_i$, there is precisely one outgoing edge at $x$ labeled by $a_i$ and precisely one incoming edge at $x$ labeled by $a_i$ (thus, in particular, an $A$-regular graph is folded).

If $\Gamma$ is an $A$-graph and $p=e_1,\dots, e_k$ is an edge-path in $\Gamma$,  then  $p$ has a label which is a word in $A \sqcup A^{-1}$ and we denote this label by $\mu(p)= \mu(e_1)\mu(e_2)\ldots \mu(e_k)$.  The definitions immediately imply:
\begin{lem}
An $A$-graph $\Gamma$ is folded if and only if the label of every reduced path in $\Gamma$ is a freely reduced word.
\end{lem}
 
\begin{df}\label{d:A-morph}
For any two $A$-graphs $\Gamma_1$ and $\Gamma_2$, a map $f: \Gamma_1 \rightarrow \Gamma_2$ is an $A$-morphism if $f$ is a graph-map  such that $\mu (e)= \mu(f(e))$.   
\end{df}

%Note that if the path $\gamma$ in an $A$-graph $\Gamma$ is already cyclically reduced (which is guaranteed, for example, if its label is a cyclically reduced word over $A$), then $\gamma=\gamma_1$ and the word over $S_T$ produced in the first part of the above proposition is already cyclically reduced.

For $F_N= F(a_1, \ldots, a_N)$ we define the \textit{standard $N$-rose} $R_N$ to be the wedge of $N$ loop-edges  each labeled by $a_1, \ldots, a_N$ respectively, at a vertex $x_0$. Then $F(A) = \pi_1(R_N, x_0)$.  

For $\Gamma$ an $A$-graph, $x \in V\Gamma$ and $\mu$ as before, we can define a map $\mu_{\#} : \pi_1(\Gamma, x) \rightarrow F(A)$ as $p \mapsto \underline{\mu(p)}$. This map is a group homomorphism. 
\begin{notation}
For $\Gamma$ an $A$-graph, $x \in V\Gamma$ we say that $(\Gamma, x)$ represents the subgroup $H:= \mu_{\#}(\pi_1(\Gamma, x)) \leq F(A)$.
\end{notation}

\begin{propdef}\label{pd:grep}\cite{St1,KM}
Let $H \leq F(A)$. Then there exists a connected, folded $A$-graph $\Gamma$ with $x_0 \in V\Gamma$ such that $\Gamma= Core(\Gamma, x_0)$ and $(\Gamma, x_0)$ represents $$H = \{\mu(p) \, | \, p \, \text{is a reduced path in} \, \Gamma \, \text{from} \, x_0 \, \text{to} \, x_0 \} \leq F(A)$$
Moreover, such a $(\Gamma, x_0)$ is unique. This graph $(\Gamma, x_0)$ is called the \emph{Stallings subgroup graph of $H$} with respect to $A$.  
\end{propdef}

If $(\Gamma,x_0)$ is the Stallings subgroup graph for $H$, then the labeling map $\mu: \pi_1(\Gamma,x_0)\to H$ is a group isomorphism. If $T\subseteq \Gamma$ is a maximal tree and $S_T=\{s_e| e \in E_+(\Gamma - T)\}$ is the dual free basis of $\pi_1(\Gamma,x_0)$, then $\mu(S_T)= \{\mu(s_e) | e \in E_+(\Gamma-T)\}$ is a free basis of $H$.

\subsection{Primitive, simple and non-filling elements}\label{s:elements}

\begin{df}[Primitive and simple elements]
In the free group $F_N$, a non-trivial element $g \in F_N$ is called
\textit{primitive} in $F_N$ if $g$ belongs to some free basis of $F_N$.

In the free group $F_N$, a non-trivial element $g \in F_N$ is called
\textit{simple} in $F_N$ if $g$ belongs to a proper free factor of $F_N$. 
\end{df}

\begin{df}[Non-filling elements]\label{df:nonfill}
An element $g\in F_N$ is said to be \emph{non-filling} in $F_N$ if there exists
a splitting of $F_N$ as an amalgamated free product $F_N=K\ast_C L$ or
as an HNN-extension $F_N=\langle K, t | t^{-1}Ct=C'\rangle$, such that
$C\le F_N$ is either trivial or a maximal cyclic subgroup, such that in
the  $F_N=K\ast_C L$ case $C\ne K, C\ne L$, and such that $g\in K$. 

An element $g\in F_N$ is said to be \emph{filling} in $F_N$ if $g$ is
not non-filling. 
\end{df}

\begin{rk}\label{r:ps}
Note that if $g \in F_N$ is primitive, then it is also simple. 
Similarly, if $g \in F_N$ is simple, then $g$ is non-filling.

Also, for elements of $F_N$ the properties of being primitive, being simple and being
non-filling  are preserved under applying arbitrary automorphisms of $F_N$.
\end{rk}

The following known key fact  relates the property of being filling in
$F_N$ to the compactification $\CVNbar$ of the projectivized Culler-Vogtmann Outer
space $\CVN$. This
compactification consists of the projective classes of all minimal
``very small'' isometric actions of $F_N$ on $\mathbb R$-trees.
An isometric action of a group $G$ on an $\mathbb R$-tree $T$ is called \emph{very small} if for every nondegenerate segment of $T$ the setwise stabilizer of that segment in $G$ is either trivial or maximal infinite cyclic in $G$, and if the  setwise stabilizer  of every tripod in $T$ is trivial. For example, the Bass-Serre tree corresponding to a splitting of $F_N$ as an amalgamated product or an HNN-extension over a maximal infinite cyclic subgroup is a very small $F_N$-tree.  See~\cite{KL,BF} for a more detailed explanation of the relevant terminology. Also, if $T$ is an $\mathbb R$-tree equipped with an isometric action of a group $G$, for $g\in G$ we denote $||g||_T:=\inf_{x\in T} d(x,gx)$; the quantity $||g||_T$ is called the \emph{translation length} of $g$ in $T$.

\begin{prop}\cite{KL,Solie}\label{prop:filling}
Let $1\ne g\in F_N$. Then the following conditions are equivalent:
\begin{enumerate}
\item The element $g$ is filling in $F_N$.
\item For every minimal very small isometric action of $F_N$ on a
  nontrivial simplicial tree $T$ we have $||g||_T>0$.
\item For every  minimal very small isometric action of $F_N$ on a
  nontrivial $\mathbb R$-tree $T$ we have $||g||_T>0$.
\end{enumerate}
\end{prop}

\begin{proof}
The proof of this statement is implicit in \cite{KL,Solie} but we
sketch the argument for completeness.

Part (3) directly implies part (2). Since the simplicial splittings that appear
in Definition~\ref{df:nonfill} are very small, part (2) also directly
implies part (1). 

To see that part (1) implies part (3), suppose that $1\ne g\in F_N$ is
filling but that there exists a minimal very small isometric action of $F_N$ on a
  nontrivial $\mathbb R$-tree $T$ we have $||g||_T=0$. Then a result
  of Bestvina and Feighn~\cite{BF} (see also a paper of Guirardel~\cite{Gui}) implies that there exists a very small
  minimal simplicial $F_N$-tree $T'$ with $||g||_{T'}=0$. Taking the
  quotient graph of groups $T'/F_N$ and collapsing all edges except
  one in this graph gives us a splitting of $F_N$ as in
  Definition~\ref{df:nonfill} such that $g$ is conjugate to a vertex
  group element for that splitting. This contradicts the assumption
  that $g$ is filling in $F_N$. Thus (1) implies (3), as required.

\end{proof}

\subsection{Whitehead Graphs}\label{sect:Whgraph}

We now describe the relationship between simple elements, primitive elements, and Whitehead graphs.

\begin{df}\label{d:whgraph}[Whitehead graph]
Let $F_N= F(A)$ be as before and let $w \in F_N$ be a nontrivial cyclically reduced word.  Let $c$ be the first letter of $w$. The word $wc$ is then freely reduced. 

The \emph{Whitehead graph of $w$} with respect to $A$, denoted by $Wh_A(w)$, is an undirected graph whose set of vertices $V(Wh_A(w))= \Upsilon$. Edges are added as follows: For $a,b \in V(Wh_A(w))$, there is an undirected edge joining $a^{-1}$ and $b$ if $ab$ or $b^{-1}a^{-1}$ occurs as a subword of $wc$.      

Note that if $\tilde w$ is a cyclic permutation of $w$ or of $w^{-1}$ then $Wh_A(w)=Wh_A(\tilde w)$.

For an arbitrary $1\ne g\in F_N$, we put $Wh_A(g):=Wh_A(w)$, where $w$ is the cyclically reduced form of $g$ in $F(A)$.

\end{df}

Recall that a \emph{cut vertex} in a graph $\Delta$ is a vertex $x$ such that $\Delta-\{x\}$ is disconnected. Note that if $\Delta$ has at least one edge and is disconnected, then $\Gamma$ does possess a cut vertex; namely any end-vertex of an edge of $\Delta$ is a cut vertex in this case. 

Generalizing a result of Whitehead, Stallings established the relationship between simple elements and Whitehead graphs~\cite{St}:

\begin{prop}\label{p:cut}~\cite{St}
Let $F_N=F(A)$, where $N\ge 2$ and let $g \in F(A)$ be a cyclically reduced word. If $g$ is simple, then the Whitehead graph $Wh_A(g)$ has a cut vertex.
\end{prop}

Notice that Remark \ref{r:ps} implies that if $g \in F(A)$ is primitive, then $Wh_A(g)$ has a cut vertex. 

\begin{rk}
Stallings' definition of Whitehead graphs differs slightly from our definition. Assume the same setting as in Definition \ref{d:whgraph}. Stallings adds an edge from $a^{-1}$ to $b$ for \textit{each} occurrence of a subword $ab$ in $wc$. Let us call the Whitehead graph of a cyclically reduced word $w$ under Stallings' definition $\Gamma$, and the corresponding graph under our definition $\Gamma_1$. It is clear that $V(\Gamma)= V(\Gamma_1)$. Further it is easily checked that $x \in V(\Gamma)$ is a cut-vertex in $\Gamma$ if and only if $x\in V(\Gamma_1)$ is a cut-vertex in $\Gamma_1$. Thus Proposition~\ref{p:cut} holds for our definition of Whitehead graphs just as well. 
\end{rk}

Finally, note that if a graph has a reduced circuit that contains all the vertices, then the graph can not have a cut vertex. This observation applies, for instance, to the Whitead graph of an element $g\in F_N$ when the string $a_N^2a_1^2a_2^2\ldots a_N^2$ occurs as a subword of a cyclically reduced form of $g$. In this case $g$ is not simple (and hence not primitive) as its Whitehead graph does not have a cut vertex. We state this fact explicitly as a corollary of Proposition~\ref{p:cut}:

\begin{cor}\label{c:np}
Let $F_N=F(A)$, where $N\ge 2$ and $A=\{a_1,\dots, a_N\}$. 
If a cyclically reduced word $w \in F(A)$ contains the subword $a_N^2a_1^2a_2^2\ldots a_N^2$ then $w$ is not simple (and hence not primitive) in $F(A)$. 
\end{cor}  

The Whitehead graph, as defined above, records the information about
two-letter subwords in the cyclically reduced form $w$ of a nontrivial
element $g\in F_N=F(A)$. There are also generalizations of the
Whitehead graph recording the information about $k$-letter  subwords
of $w$, where $k\ge 2$ is a fixed integer. These generalizations are
commonly known as ``Rauzy graphs'' or ``initial graphs'' and naturally
occur in the study of geodesic currents on free groups~\cite{Ka1,Ka2,Kapovich}.

We do not formally define these ``level $k$'' versions of the
Whitehead graph in this paper, because we only need the following
specific statement related to the $k=3$ case:

\begin{prop}\label{p:ChM}\cite{ChM}
Let $F_N=F(A)$, where $N\ge 2$ and $A=\{a_1,\dots, a_N\}$. 
Let $w$ be a nontrivial cyclically reduced word  in $F(A)$ such that
for every freely reduced word $v\in F(A)$ with $|v|=3$ the word $v$
occurs as a subword of a cyclic permutation of $w$ or of $w^{-1}$.

Then $w$ is filling in $F_N$ (and, in particular, $w$ is non-simple
and non-primitive in $F_N$).
\end{prop}

\section{Primitivity, Simplicity, and Non-Filling Index Functions}

In 1949 Marshall Hall Jr. proved in~\cite{MH} that any finitely generated subgroup of a free group $F_N$ is a free factor of a finite index subgroup of $F_N$. We state the result in a more precise form, as stated in~\cite{St1}:

\begin{prop}\label{p:MH}\cite{St1}
Let $\alpha_1, \ldots, \alpha_k, \beta_1, \ldots, \beta_l$ be elements of a free group $F_N$. Let $S$ be the subgroup of $F_N$ generated by $\{\alpha_1, \ldots, \alpha_k\}$. Suppose $\beta_i \notin S$ for $i=1, \ldots, l$. Then there exists a subgroup $S'$ of finite index in $F_N$, such that $S\subset S'$, $\beta_i \notin S'$ for $i=1, \ldots, l$, and there exists a free basis of $S'$ having a subset that is a free basis of $S$.
\end{prop}

If we pick $g \neq 1 \in F_N$ and apply the above result to the
infinite cyclic subgroup $S=\langle g\rangle$, we get that there must
exist a finite index subgroup $S'$ of $F_N$ such that $g$ is a
primitive element in $S'$ (and hence $g\in S'$ is non-simple and
non-filling in $S'$).

This fact motivates the following definition:

\begin{df}\label{d:pif}[Primitivity, simplicity and non-filling
  indexes] 

Let $N\ge 2$ be an integer and let $F_N$ be a free group of rank $N$.
Let $1\ne g\in F_N$.  

Define the \emph{primitivity index} $\dpr(g)= \dpr(g, F_N)$ of $g$ in $F_N$ to be the
smallest possible index for a subgroup $L\le F_N$ containing $g$ as a
primitive element.

Define the \emph{simplicity index} $\dsi(g)= \dsi(g, F_N)$ to be the
smallest possible index for a subgroup $L\le F_N$ containing $g$ as a
simple element.

Finally, define the  \emph{non-filling index} $\dfi(g)= \dfi(g, F_N)$ to be the
smallest possible index for a subgroup $L\le F_N$ containing $g$ as a
non-filling element.

\end{df}

As noted above, Proposition~\ref{p:MH} implies that for every
nontrivial $g\in F_N$ we have $\dfi(g)\le \dsi(g)\le \dpr(g)<\infty$.

\begin{df}[Primitivity, simplicity and non-filling index functions]
Let $F_N$ be a free group of rank $N\ge 2$ and let $A$ be a free basis of $F_N$.
For any $n \geq 1$ define the \emph{primitivity index function} for $F_N$ as  
\[
\fp(n)=\fp(n;F_N):= \max_{\substack {1 \leq |g|_A \leq n, \,  g \neq 1 \\
                                                 g \, \text{not a proper power in } F_N}} \dpr(g)
\]

Similarly, for $n\ge 1$ 
 define the \emph{simplicity index function} for $F_N$ as  
\[
\fs(n)=\fs(n;F_N):= \max_{\substack {1 \leq |g|_A \leq n, \,  g \neq 1 \\
                                                 g \, \text{not a proper power in } F_N}} \dsi(g)
\]                     

Finally, for for $n\ge 1$ 
 define the \emph{non-filling index function} for $F_N$ as  
\[
\ff(n)=\ff(n;F_N):= \max_{\substack {1 \leq |g|_A \leq n, \,  g \neq 1 \\
                                                 g \, \text{not a proper power in } F_N}} \dfi(g)
\]                  																								
\end{df}
It is easy to see that the definitions of $\fp(n;F_N)$, $\fs(n;F_N)$
and $\ff(n;F_N)$ do not depend on the choice of a free basis $A$ of $F_N$.
Note that $\fp(n)$ is the smallest monotone non-decreasing function
such that for every non-trivial root-free $g \in F_N$ we have $\dpr(g)
\leq \fp(|g|_A)$; similar reformulations hold for $\fs(n)$ and $\ff(n)$.  We recall the following well-known fact, which is Lemma~8.10 in~\cite{KM}:

\begin{lem}\label{l:co}
Let $\Gamma$ be a finite folded $A$-graph. Then there exists a finite folded $A$-regular graph $\Gamma'$ such that $\Gamma$ is a subgraph of $\Gamma'$ and such that $V\Gamma = V\Gamma'$.
\end{lem}

\begin{prop}\label{p:comp}
For every non-trivial cyclically reduced word $w \in F(A)$ of length
$n$, there exists a finite index subgroup $H \leq F(A)$ of index $n$
such that $w \in H$ is primitive in $H$.   
\end{prop}
\begin{proof}
Take the word $w$ of length $n$ and write it on a circle of simplicial length $n$. Pick a vertex $x$ as the base vertex. Call this graph $(\Gamma_w, x)$. By Lemma \ref{l:co} we can complete this graph to a finite cover $(\Gamma_w',x)$ of the $N$-rose without adding any extra vertices. Thus $(\Gamma_w',x)$ has $n$ vertices and represents a subgroup $H$ of $F_N$ of index precisely $n$. The fact that $w$ is realized as the label of a simple closed curve in $(\Gamma_w',x)$ implies that $w$ is a primitive element in $H$. It is clear that $w \in H$ by definition of $H$. Note that since $(\Gamma',x)$ has no extra vertices, a maximal tree $T$ of $(\Gamma, x)$ consists of all but one edge of the simple closed curve representing $w$. Let $e \in E_+\Gamma' - T$. Then $\mu(s_e)=w$ and hence $w$ is primitive. See Figure \ref{Comp} for a pictorial proof.

\begin{figure}
\centering
\includegraphics[trim=0cm 4cm 0cm 8cm, clip=true, totalheight=0.4\textheight, angle=0]{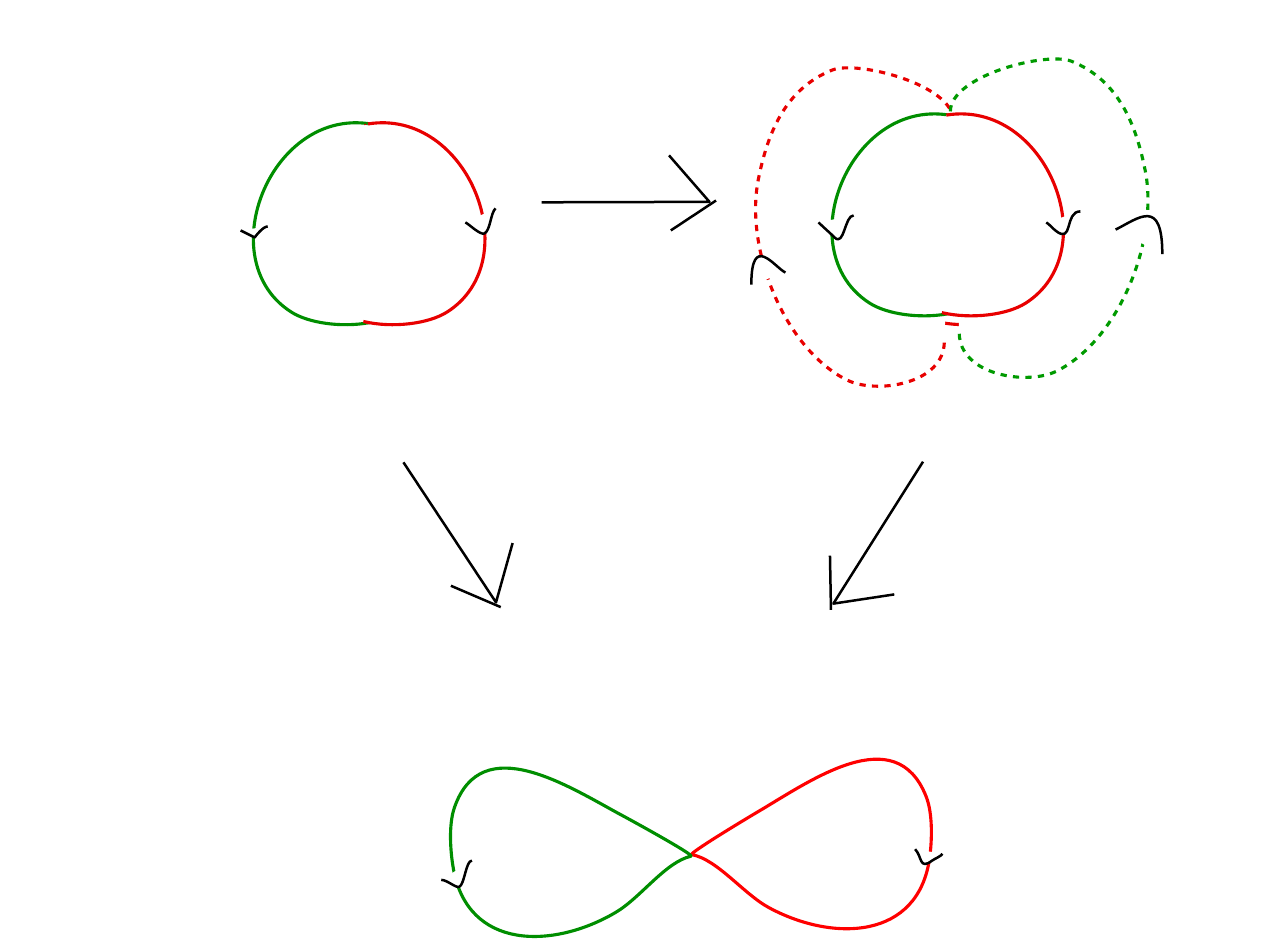}
\caption{Proof by Picture for Proposition \ref{p:comp}\label{Comp}}
\end{figure}

\end{proof}

Proposition \ref{p:comp}, together with the definitions, directly implies:
\begin{lem}\label{l:aut}
Let $N\ge 2$ and let $F_N$ be free of rank $N$. Then the following
hold:

\begin{enumerate}
\item If $1\ne g\in F_N=F(A)$ then 
\[
\dfi(g)\le \dsi(g)\le \dpr(g) \leq ||g||_A \leq |g|_A = n.\]
\item For every $n\ge 1$ we have  
\[
\ff(n)\le \fs(n)\le \fp(n) \leq n.
\] 
\item Let $1\ne g \in F_N$ and let $\alpha \in$ Aut$(F_N)$. Then
  $\dpr(g)=\dpr(\alpha(g))$, $\dsi(g)= \dsi(\alpha(g))$ and $\dfi(g)=
  \dfi(\alpha(g))$.
\item If $1\ne g\in F_N$ and $k\ge 1$ is an integer, then
  $\dsi(g^k)\le \dsi(g)$ and $\dfi(g^k)\le \dfi(g)$.
\end{enumerate}

\end{lem}

In particular, part (3) of the above lemma shows that for $g_1, g_2$
conjugate non-trivial elements of $F_N$, we have $\dpr(g_1)=
\dpr(g_2)$, $\dsi(g_1)= \dsi(g_2)$ and $\dfi(g_1)= \dfi(g_2)$.

As noted above, if $1\ne g\in F_N$ and $k\ge 1$ is an integer, then
$\dsi(g^k)\le \dsi(g)$ and $\dfi(g^k)\le \dfi(g)$. However, the function $\dpr(g)$ does not behave well under taking powers, as demonstrated by the following lemma:

\begin{lem}\label{lem:powers}
For any $a_i \in \{a_1, \ldots, a_N \}$, and any positive integer $n$, $\dpr(a_i^n)=n$.
\end{lem}
\begin{proof}
As noted above, for every nontrivial $g\in F_N$ we have $\dsi(g)\le \dpr(g)\le ||g||_A$. Thus $\dpr(a_i^n)\le ||a_i^n||_A=n$.
We need to show that $\dpr(a_i^n)\ge n$.

Let $d=\dpr(a_i^n)$ and let  $H\le F_N$ be a subgroup of index $d$ such that $a_i^n\in H$ and that $a_i^n$ is a primitive element of $H$. Let $(\Gamma,\ast)$ be the $d$-fold cover of $R_N$ corresponding to $H$, so that for the covering map $p:\Gamma\to R_N$ have $\pi_1(\Gamma,\ast)\cong H$ and $p_\#=\mu:\pi_1(\Gamma,\ast)\to H\le F_N=\pi_1(R_N,x_0)$ is an isomorphism. 

The fact that $a_i^n\in H$ implies that there exists a reduced closed path $\gamma$ from $\ast$ to $\ast$ in $\Gamma$ with $\mu(\gamma)=a_i^n$.
Since $a_i^n$ is primitive in $H$, the element $\gamma$ is primitive in $\pi_1(\Gamma,\ast)$.

Since $a_i^n$ is cyclically reduced, the closed path $\gamma$ is also cyclically reduced. We claim that $\gamma$ is a simple closed path in $\Gamma$.
Indeed, suppose not. Then $\gamma=\gamma_1^k$ where $k\ge 2$ and where $\gamma_1$ is a simple closed path at $\ast$ in $\Gamma$ with label $a_i^{n/k}$. 
Therefore $\gamma$ is a proper power in $\pi_1(\Gamma,\ast)$, which contradicts the fact that $\gamma$ is primitive in $\pi_1(\Gamma,\ast)$.
Thus indeed $\gamma$ is a simple closed path in $\Gamma$ with label $a_i^n$. This means that the full $p$-preimage of the $i$-th petal of $R_N$, labeled $a_i$, in $\Gamma$ consists of $\ge n$ distinct topological edges.  Therefore the degree $d$ of the cover $p:\Gamma\to R_N$ satisfies $d\ge n$.

Thus $d=\dpr(a_i^n)\ge n$. Since we already know that $\dpr(a_i^n)\le n$, it follows that $\dpr(a_i^n)=n$, as required.
\end{proof}

Avoiding the bad behavior of $\dpr(g)$ under taking powers of $g$,
demonstrated by Lemma~\ref{lem:powers}, is the main reason why in
Definition~\ref{d:pif} we take the maximum over all root-free
nontrivial elements $g\in F_N$ with $|g|_A\le n$ rather than over all
nontrivial $g\in F_N$ with $|g|_A\le n$.

\section{Algorithmic computability of $\dpr(g)$, $\dsi(g)$, and $\dfi(g)$}\label{sect:alg}

In this section we will establish algorithmic computability of $\dpr(g)$, $\dsi(g)$, and $\dfi(g)$. Consequently, we will also establish the algorithmic computability of $\fp(n)$, $\fs(n)$, and $\ff(n)$.
\color{black}

We first need to recall some basic definitions and facts related to Whitehead automorphisms and Whitehead's algorithm. We only briefly cover this topic here and refer the reader for further details to~\cite[pp. 30-35]{LySch} and to \cite{MS,KSS,Kapovich,RVW} for some of the recent developments.  As before, $F_N=F(A)= F(a_1, \ldots, a_N)$ is the free group of rank $N\ge 2$ with a free basis $A=\{a_1,\dots, a_N\}$.

\begin{df}[Whitehead automorphisms]
A \emph{Whitehead automorphism} $\tau$ of $F_N=F(A)$ with respect to $A$ is an automorphism $\tau$ of $F(A)$ of one of the following types:
\begin{enumerate}
\item There exists a permutation $t$ of $\Upsilon= A\sqcup A^{-1}$ such that $\tau|_{\Upsilon}=t$. In this case $\tau$ is called a \emph{relabeling automorphism} or a \emph{Whitehead automorphism of the first kind}.
\item There exists an element $a\in \Upsilon$ which we call the \emph{multiplier} such that for any $x\in \Upsilon$, $\tau(x) \in \{x, xa, a^{-1}x, a^{-1}xa\}$. In this case $\tau$ is called a \emph{Whitehead automorphism of the second kind}. 
\end{enumerate}
\end{df}

Note that since $\tau \in Aut(F(A))$, if $\tau$ is a Whitehead automorphism of the second kind with multiplier $a$, then $\tau(a)=a$. Also for any $a \in \Upsilon$, the inner automorphism corresponding to conjugation by $a$ is a Whitehead automorphism of the second kind.  

\begin{df}[Automorphically minimal and Whitehead minimal elements]
An element $g\in F(A)=F_N$ is \emph{automorphically minimal} in $F(A)$ with respect to a basis $A$ of $F_N$ if, for every $\phi\in Aut(F(A))$ we have $||g||_A\le ||\phi(g)||_A$. 

An element $g\in F(A)$ is \emph{Whitehead minimal} in $F(A)$ with respect to a free basis $A$ if, for every Whitehead automorphism $\tau$ of $F(A)$ we have $||g||_A\le ||\tau(g)||_A$.
For an element $g\in F(A)$ we say that $\tilde g\in F(A)$ is a \emph{Whitehead minimal form} of $g$ with respect to $A$ if $\tilde g$ is Whitehead minimal with respect to $A$ and there exists an automorphism $\phi\in \Aut(F(A))$ such that $\phi(g)=\tilde g$. 

\end{df}

Note that neither Whitehead automorphisms of the first kind nor inner automorphisms change the cyclically reduced length of an element. 
  
  The following proposition summarizes the key known facts regarding Whitehead's algorithm (see~\cite{Wh} for the original proof by Whitehead and see \cite[Proposition 4.17]{LySch} for a modern exposition): 
  
\begin{prop}[Whitehead's Theorem] \label{p:wa} 
Let $N\ge 2$ and let $F_N=F(A)$ be free of rank $N$ with a free basis $A$. Then:
\begin{enumerate}
\item An element $g\in F(A)$ is automorphically minimal in $F(A)$ with respect to a basis $A$  if and only if $g$ is Whitehead minimal in $F(A)$ with respect to $A$. (Hence $g\in F(A)$ is not automorphically minimal with respect to $A$ if and only if there exists a Whitehead automorphism $\tau$ such that $||\tau(g)||_A<||g||_A$). 
\item Whenever $u,v\in F(A)$ are Whitehead minimal with respect to $A$ such that the orbits $Aut(F(A))u=Aut(F(A))v$ (so that, in particular, $||u||_A=||v||_A$), then there exists a sequence of Whitehead automorphisms $\tau_1,\dots, \tau_m$ of $F(A)$ with respect to $A$ such that $\tau_m...\tau_1(u)=v$ and that $||\tau_i...\tau_1(u)||_A=||u||_A$ for $i=1,...,m$.
\end{enumerate}
\end{prop}

Note that part (2) of Proposition~\ref{p:wa}  holds even if $u,v$ are conjugate in $F(A)$ since conjugation by an element of $A^{\pm 1}$ is a Whitehead automorphism. 

\subsection{Algorithmic computability of $\dpr(g)$ and $\dsi(g)$}
The following useful lemma explicitly states the relationship between primitivity, simplicity and Whitehead minimality:

\begin{lem}\label{l:mps} Let $1\ne w\in F(A)=F_N$. 
\begin{enumerate}
\item $w$ primitive in $F(A)$ if and only if every (equivalently, some) Whitehead minimal form $\widetilde{w}$ of $w$ has $||\widetilde{w}||_A= 1$. 
\item $w$ is simple in $F(A)$ if and only if some Whitehead minimal form $\widetilde{w}$ of $w$ misses an $a_{i}^{\pm 1}$.
\item $w$ is simple in $F(A)$ if and only if every Whitehead minimal cyclically reduced form $\widetilde{w}$ of $w$ misses an  $a_{i}^{\pm 1}$. 
\end{enumerate}

\end{lem}

\begin{proof}

Part (1) of the lemma is well-known and follows directly from Proposition~\ref{p:wa}.

If some Whitehead minimal form $\widetilde{w}$ of $w$ misses an $a_{i}^{\pm 1}$, then $w$ is simple in $F(A)$ as $w \in F(B)$ where $B = A-\{a_i\}$ and $F(B)$ is a proper free factor of $F(A)$. 

Conversely, suppose that $w$ is simple in $F(A)$. Then there exists an automorphism $\phi$ of $F(A)$ such that the cyclically reduced form $\widehat{w}$ of $\phi(w)$ misses $a_N^{\pm 1}$.

{\bf Claim 1.} We claim that some Whitehead minimal form of $\widehat{w}$ also misses $a_N^{\pm 1}$.

We prove this claim by induction on $||\widehat{w}||_A$.
If $||\widehat{w}||_A=1$, then the claim clearly holds.  Suppose now that $||\widehat{w}||_A=m>1$ and that the claim has been established for all nontrivial cyclically reduced words in $F(a_1,\dots, a_{N-1})$ of length $\le m-1$. 

 If $\widehat{w}$ is already Whitehead minimal in $F(A)$ then we are done as the claim holds in this case. 
 
 If $\widehat{w}$ is not Whitehead minimal in $F(A)$ then there exists a Whitehead automorphism $\tau$ of $F(A)$ such that $||\tau(\widehat{w})||_A < ||\widehat{w}||_A$.   Note first that since the cyclically reduced length of $\widehat{w}$ changes under $\tau$, we must have that $\tau$ is a Whitehead automorphisms of the second kind that is not an inner automorphism.

 Let $a \in \Upsilon= A \sqcup A^{-1}$ be the multiplier of $\tau$. If $a=a_N^{\pm 1}$,  since $\widehat{w}$ is a cyclically reduced word in $F(A)$ that misses the letter $a_N^{\pm 1}$, the definition of a Whitehead automorphism implies that there can be no cancellation in $\tau(\widehat{w})$ between the letters $\{a_1, \ldots, a_{N-1}\}$ when a cyclically reduced form of $\tau(\widehat{w})$ is computed. Hence $||\tau(\widehat{w})||_A\geq ||\widehat{w}||_A$, contrary to the fact that $||\tau(\widehat{w})||_A < ||\widehat{w}||_A$. 
 Therefore $a \in \{a_1, \ldots, a_{N-1}\}^{\pm 1}$.  We then define a Whitehead automorphism $\tau'$ of $F(a_1,\dots, a_{N-1})$ with respect to $\{a_1,\dots, a_{N-1}\}$ as  $\tau' = \tau|_{\{a_1, \ldots, a_{N-1}\}}$.   Hence $\tau(\widehat{w})= \tau'(\widehat{w})$. Thus $\tau(\widehat{w})$ still misses $a_N^{\pm 1}$ and $||\tau(\widehat{w})||_A < ||\widehat{w}||_A=m$.  
 Applying the inductive hypothesis to $\tau(\widehat{w})$, we conclude that some Whitehead minimal form $\widetilde{w}$ of  $\tau(\widehat{w})$ in $F(A)$ misses $a_N^{\pm 1}$. Then $\widetilde{w}$  is also a Whitehead minimal form of $\hat w$, and Claim~1 is verified. 
 
Thus we have established part (2) of the lemma. 
 
To see that part (3) holds, note that if every Whitehead minimal cyclically reduced form $\widetilde{w}$ of $w$ misses an  $a_{i}^{\pm 1}$ then $w$ is simple in $F(A)$. 

Now suppose $w$ is simple in $F(A)$. From (2) we know that there is a $\widetilde{w}$ Whitehead minimal cyclically reduced form of $w$ that misses $a_N^{\pm 1}$. Let $w'$ be another Whitehead minimal cyclically reduced form of $w$ in $F(A)$. Then $Aut(F(A))w'=Aut(F(A))\widetilde{w}$, and so by part (2) of Proposition~\ref{p:wa}, there exists a sequence of Whitehead automorphisms $\tau_1,\dots, \tau_m$ of $F(A)$ with respect to $A$ such that $\tau_m...\tau_1(\widetilde{w})=w'$ and that $||\tau_i...\tau_1(\widetilde{w})||_A=||w'||_A$ for $i=1,...,m$. 

For $j=0,1,\dots, m$ denote $w_j=\tau_j...\tau_1(\widetilde{w})$, where $w_0=\widetilde{w}$.

{\bf Claim 2.} We claim that for each $j=0,\dots, m$ the cyclically reduced form of $w_j$ misses some $a_i^{\pm 1}$. 

We will establish Claim 2 by induction on $j$.

If $j=0$ then $w_0=w$ and there is nothing to prove. Suppose now that $j\ge 1$ and that the claim has been verified for $w_{j-1}$. 

Thus the cyclically reduced form of $w_{j-1}$ misses some $a_i^{\pm 1}$.  If $\tau_j$  is a Whitehead automorphism of the first kind, it is clear that the cyclically reduced form of $\tau_j(w_{j-1})=w_j$ still misses some $a_k^{\pm 1}$ (this $a_k^{\pm 1}$ is not necessarily $a_i^{\pm 1}$).  Suppose now that $\tau_j$  is a Whitehead automorphism of the second kind. 
The restriction that $||\tau_j(w_{j-1})||_A=||w_{j-1}||$ forces the condition that either $\tau_j(w_{j-1})$ is equal to $w_{j-1}$ after cyclic reduction, or else $\tau_j$ is a Whitehead automorphism of the second kind with multiplier $a \in B\sqcup B^{-1}$ where $B= \{x \in A \sqcup A^{-1}|x  \text{ occurs in the cyclically reduced form of} \, w_{j-1}\}$ (in particular $a \neq a_i^{\pm 1}$). 
In both cases we see that the cyclically reduced form of $w_j$ still misses $a_i^{\pm 1}$, as required. This completes the inductive step and the proof of Claim~2.

Applying Claim 2 with $j=m$ shows that the cyclically reduced form of $w'=w_m$ misses some $a_i^{\pm 1}$, and part (3) of the lemma is proved.
\end{proof}

\begin{prop}\label{p:H}  Let $1\ne g\in H\le F(A)$, where $H$ is a proper free factor of $F(A)$. Then the following hold:
\begin{enumerate}
\item The element $g$ is primitive in $H$ if and only if $g$ is primitive in $F_N$.
\item There is an algorithm which decides, given $g\in F(A)$,  whether or not $g\in F(A)$ is primitive.
\item There is an algorithm which given $g\in F(A)$,  whether or not $g\in F(A)$ is simple.
\end{enumerate}
\end{prop}

\begin{proof}

We first prove part (1).  The ``only if" direction is obvious. Thus we assume that $g\in H$ is primitive in $F_N$.

Let $K\le F_N$ be such that $F_N= H \ast K$. Let $\mathcal{B}_H= \{ h_1, \ldots, h_l \}$ be a free basis for $H$, and $\mathcal{B}_K= \{k_1, \ldots, k_m\}$ be a free basis for $K$. Then $\mathcal{B}_F= \{h_1, \ldots, h_l, k_1, \ldots, k_m \}$ is a free basis for $F_N$ (here $l+m=n$). 

Since $g\in H$, then $g$ is a freely reduced word over $\mathcal B_H$, with cyclically reduced form $w$. 
We prove that $g$ is primitive in $H$ by induction on the length $m$ of $w$.

If $w$ has length $1$, then $g$ is primitive in $H$, as required. If $w$ has length $m>1$, then the fact that $w$ is primitive in $F_N$ implies that $w$ is not Whitehead minimal in $F_N$ with respect to the free basis $\mathcal{B}_F$ of $F_N$. Hence there exists a Whitehead automorphism $\tau$ of $F_N$ with respect to $\mathcal{B}_F$ such that $||\tau(w)||_{\mathcal{B}_F}<m$.  (Note that at this point we do not yet know that $\tau(w)\in H$ since $\tau$ is a Whitehead automorphism of $F_N$, and not of $H$).

By the same argument as in the proof of Lemma \ref{l:mps}, we see that there exists a Whitehead automorphism $\tau'$ of $H=F(\mathcal{B}_H)$ such that $\tau'(w)=\tau(w)$. Then $\tau(w)=\tau'(w)\in H$ is primitive in $F_N$ with $||\tau(w)||_{\mathcal{B}_F}<m$. Therefore by the inductive hypothesis the element $\tau(w)=\tau'(w)$ is primitive in $H$. Since $\tau'\in Aut(H)$, it follows that $w$ is also primitive in $H$, as required. Thus part (1) of the proposition holds.

To prove parts (2)  and (3) for $g \in F(A)= F(a_1, \ldots, F_N)$, let $\widetilde{g}$ be a Whitehead minimal form of $g$ in $F(A)$ (such $\widetilde{g}$ exists by Proposition~\ref{p:wa}. By part (1) of Lemma \ref{l:mps}, $||\widetilde{g}||_A= 1$ if and only if $g$ is primitive in $F(A)$.  By part (3) of Lemma \ref{l:mps}, $\widetilde{g}$ misses some $a_i^{\pm 1}$ if and only if $w$ is simple in $F(A)$. 
\end{proof}

\begin{rk}
The algorithm described in part (2) of Proposition~\ref{p:H} is due to Whitehead~\cite{Wh}.
The first algorithms for deciding whether an element of $F_N$ is simple in $F_N$ were provided by Stallings~\cite{St} and Stong~\cite{Stong}  in 1990s. Their algorithms are somewhat different from the algorithm given in part (3) of Proposition~\ref{p:H} above, but they are also based on using Whitehead's algorithm.
\end{rk}

\begin{df}[Principal quotient]\label{d:pq}
Following the terminology of \cite{KM}, for a finite connected $A$-graph $\Gamma_1$ and a folded $A$-graph $\Gamma_2$, we say that $\Gamma_2$ is a \emph{principal quotient} of $\Gamma_1$ if there exists a surjective $A$-morphism $\Gamma_1\to\Gamma_2$. 
\end{df}
\begin{df}\label{d:cw}
Let $w\in F_N=F(A)$ be a nontrivial cyclically reduced word. We denote by $C_w$ the $A$-graph which is a simplicial circle subdivided into $n=||w||_A$ topological edges, such that the label of the closed path of length $n$ corresponding to going around this circle once from some vertex $\ast$ to  $\ast$ is the word $w$.

By definition, the graph $C_w$ has a distinguished base-vertex $\ast$.  Thus a principal quotient of $C_w$ also come equipped with a distinguished base-vertex. We say that $(\Gamma,x)$ is a \emph{principal quotient} of $C_w$ if $\Gamma$ is a finite connected folded $A$-graph, if $x\in V\Gamma$ and if there exists a surjective $A$-morphism $f:C_w\to\Gamma$ such that $f(\ast)=x$.
\end{df}

Note that if $(\Gamma,x)$ is a principal quotient of $C_w$, then there exists a unique path $\gamma_{w,x}$ in $\Gamma$ starting with $x$ and with label $w$, and, moreover, this path is closed and passes through every topological edge of $\Gamma$.

The following lemma is an immediate corollary of the definitions:
\begin{lem}  The following hold:
\begin{enumerate}
\item  Let $\Gamma_1$ be a finite connected $A$-graph and $\Gamma_2$ be a finite folded $A$-graph. Then $\Gamma_2$ is a principal quotient of $\Gamma_1$ if and only if $\Gamma_2$ can be obtained from $\Gamma_1$ by the following procedure: choose some partition $V\Gamma_1=V_1\sqcup \dots \sqcup V_m$ (with all $V_i\ne \emptyset$), then for each $i=1,\dots, m$ collapse $V_i$ to a single vertex to get an $A$-graph $\Gamma_1'$, and then fold the graph $\Gamma_1'$ to obtain $\Gamma_2$.
 
\item  If $w\in  F_N=F(A)$ is a nontrivial cyclically reduced word and $\Gamma$ is a finite connected folded $A$-graph, then $\Gamma$ is a principal quotient of $C_w$ if and only if $\Gamma$ is a core graph and there exists a closed path $\gamma_w$ in $\Gamma$ with label $w$ such that $\gamma_w$ passes through every topological edge of $\Gamma$.
\end{enumerate}
\end{lem}

A priori it is unclear that the functions $\fp(n)$ and $\fs(n)$ are even computable for a given $F_N$. We now give an algorithm that  calculates $\dpr(g)$ and $\dsi(g)$ for any non-trivial $g$. This would then show that the functions $\fp(n)$ and $\fs(n)$ are indeed algorithmically computable.

\begin{df}\label{df:G}
Let $1\ne g\in F_N=F(A)$ and let $w\in F(A)$ be the cyclically reduced form of $g$. 
We denote by $\mathcal G_0(w)$ the set of all finite connected folded basepointed $A$-graphs $(\Gamma,x)$ such that there exists a closed path $\gamma$ from $x$ to $x$ labeled $w$ with the property that $\gamma$ passes through every topological edge of $\Gamma$ at least once and such that either the labeling map $\Gamma\to R_N$ is not a covering (that is, there exists a vertex of $\Gamma$ of degree $<2N$), or the labeling map $\Gamma\to R_N$ is a covering and the element $\gamma\in \pi_1(\Gamma,x)$ is simple in $\pi_1(\Gamma,x)$. 

We denote by $\mathcal G(w)$ the set of all finite connected folded basepointed $A$-graphs $(\Gamma,x)$ such that there exists a closed path $\gamma$ from $x$ to $x$ labeled $w$ with the property that $\gamma$ passes through every topological edge of $\Gamma$ at least once and such that the element $\gamma\in \pi_1(\Gamma,x)$ is primitive in $\pi_1(\Gamma,x)$. 
\end{df}
Let $(\Gamma,x)\in \mathcal G(w)$ or $(\Gamma,x)\in \mathcal G_0(w)$. Since $w$ is cyclically reduced and $\gamma$ passes through every topological edge of $\Gamma$ at least once, every vertex of  $\Gamma$ has degree $\ge 2$, so that $\Gamma$ is a core graph.

Note further that the condition that $\gamma$ is simple in $\pi_1(\Gamma,x)$ is equivalent to the condition that $w$ is simple in the subgroup $H \leq F_N$ represented by $(\Gamma,x)$. This follows from the fact that the labeling map gives an isomorphism $\mu: \pi_1(\Gamma,x)\to H$, with $\mu(\gamma)=w$.

We recall the following basic fact:
\begin{lem}[\cite{KM}, p.13] \label{l:ff} Let $\Gamma$ be a folded connected $A$-graph and let $\Gamma'$ be a connected subgraph of $\Gamma$. Let $\ast$ be a vertex of $\Gamma'$. If $H' \leq F(A)$ is the subgroup represented by $(\Gamma',\ast)$ and $H$ is the subgroup represented by $(\Gamma, \ast)$, then $H'$ is a free factor of $H$.
\end{lem}

\begin{rk} \label{r:ff}
In the setting of Lemma~\ref{l:ff}, $\pi_1(\Gamma',\ast)$ is a free factor of $\pi_1(\Gamma,\ast)$.
\end{rk}

\begin{prop}\label{p:criterion}
Let $1\ne g\in F_N=F(A)$ and let $w\in F(A)$ be the cyclically reduced form of $g$. Then the following hold:
\begin{enumerate}
\item The number $\dpr(g)$ equals to the minimum of $\#V\Gamma$, taken over all $(\Gamma,x)\in \mathcal G(w)$.

\item  The number $\dsi(g)$ equals to the minimum of $\#V\Gamma$, taken over all $(\Gamma,x)\in \mathcal G_0(w)$.
\end{enumerate}
\end{prop}

\begin{proof}
We give a proof of part (2). The proof of part (1) is very similar in nature. However, it additionally involves using part (1) of Proposition \ref{p:H} to prove one of the inequalities. For $1\ne g\in F_N=F(A)$ and $w\in F(A)$ the cyclically reduced form of $g$, let $\displaystyle \overline{\dsi}(g)= \min_{(\Gamma,x) \in \mathcal G_0(w)} \#V\Gamma$. First suppose that $H \leq F_N$ such that $[F_N:H]= \dsi(g)= \dsi(w)$, and that $w \in H$ is simple in $H$. Let $(\Gamma,x)$ be the graph representing $H$ as in Proposition-Definition \ref{pd:grep}. We have that $\#V\Gamma=\dsi(w)$. Since $w \in H$, there exists a path $\gamma$ from $x$ to $x$ in $\Gamma$ with label $w$. Also since $w \in H$ is simple in $H$, $\gamma \in \pi_1(\Gamma,x)$ is simple in $\pi_1(\Gamma,x)$. Let $\Gamma' \subseteq \Gamma$ be the subgraph spanned by $\gamma$. Then $\gamma$ is a path from $x$ to $x$ in $\Gamma'$ that passes through every topological edge in $\Gamma'$ at least once. If $\Gamma'=\Gamma$, then the labeling map $\Gamma'\to R_N$ is a covering. Since $\gamma$ is simple in $\Gamma=\Gamma'$, we have $(\Gamma',x) \in \mathcal G_0(w)$. Since $\#V\Gamma'= \#V\Gamma=\dsi(g)$, we have that $\overline {\dsi}(g) \leq \dsi(g)$. If $\Gamma' \neq \Gamma$, then $\#V\Gamma' \leq \#V\Gamma$ and $\#E\Gamma - \#E\Gamma' \geq 1$. From Remark \ref{r:ff}, $(\Gamma',x)$ is a proper free factor of $(\Gamma,x)$. In this case the labeling map $\Gamma'\to R_N$ is not a covering and $(\Gamma',x) \in \mathcal G_0(w)$. Thus $\overline {\dsi}(g) \leq \dsi(g)$. 

Conversely suppose that $(\Gamma,x) \in \mathcal G_0(w)$ with $\#V\Gamma=\overline{\dsi}(g)$. Let $\gamma$ be the closed path from $x$ to $x$ labeled by $w$ such that $\gamma$ passes through every topological edge of $\Gamma$ at least once. If the labeling map $\Gamma\to R_N$ is a covering then $\gamma \in \pi_1(\Gamma,x) $ is simple in $\pi_1(\Gamma,x)$ by definition of $\mathcal G_0(w) $. Let $H$ be the subgroup represented by $(\Gamma,x)$. $H$ is then a subgroup of $F_N$ of index $\overline{\dsi}(g)$ with $w\in H$ and $w$ simple in $H$. Hence $\dsi(g)=\dsi(w) \leq \overline{\dsi}(g)$. If the labeling map $\Gamma\to R_N$ is not a covering, we use Lemma \ref{l:co} to complete $(\Gamma,x)$ to a finite cover $(\widehat{\Gamma},x)$ of $R_N$ without adding any extra vertices and by adding at least one edge. Again from Remark \ref{r:ff}, $(\Gamma,x)$ is a proper free factor of $(\widehat{\Gamma},x)$. Hence $\gamma \in \pi_1(\widehat{\Gamma}, x)$ is simple in $\pi_1(\widehat{\Gamma}, x)$. Let $H$ be the subgroup represented by $(\widehat{\Gamma},x)$. We have shown that $w \in H$ is simple in $H$. Since $\#V\widehat{\Gamma}= \#V\Gamma= \overline{\dsi}(g)$, we see that $\dsi(g) \leq \overline {\dsi}(g)$.     
\end{proof}

We can now prove:

\begin{thm}\label{t:algorithm} 
Let $F_N=F(A)$, where $N\ge 2$ and where $A=\{a_1,\dots, a_N\}$ is a free basis of $F_N$. Then:
\begin{enumerate}
\item There exists an algorithm that, given $1\ne g\in F_N$, computes $\dpr(g)$ and $\dsi(g)$.
\item There exists an algorithm that, for every $n\ge 1$ computes $\fp(n)$ and $\fs(n)$
\end{enumerate}
\end{thm}
\begin{proof}

Let $1\ne g\in F_N$ and let $w$ be the cyclically reduced form of $g$.  
Note that a finite connected folded base-pointed $A$-graph $(\Gamma,x)$ admits a closed path $\gamma$ from $x$ to $x$ labeled $w$ and passing through every topological edge of $\Gamma$ at least once if and only if $(\Gamma,x)$ is a principal quotient of $C_w$ with $x$ being the image of the base-vertex $\ast$ of $C_w$.

Therefore we can algorithmically find all the graphs in $\mathcal G_0(w)$ as follows: 
List all partitions on $VC_w$. For each partition of $VC_w$ as a disjoint union of nonempty subsets $V_1,\dots V_m$, collapse $V_i$ to a single vertex for $i=1,\dots, m$, and fold the resulting graph to obtain a principal quotient $(\Gamma,x)$ of $C_w$, with $x$ being the image of the base-vertex $\ast$ of $C_w$. Let $\gamma$ be the path from $x$ to $x$ in $\Gamma$ labeled $w$ (so that, by construction, $\gamma$ passes through every topological edge of $\Gamma$ at least once). Then check whether the labeling map $\Gamma\to R_N$ is a covering, that is, whether it is true that every vertex of $\Gamma$ has degree $2N$. If $\Gamma\to R_N$ is not a covering, the graph $(\Gamma,x)$ belongs to $\mathcal G_0(w)$. If $\Gamma\to R_N$ is a covering, check, using the algorithm from part (3) of Proposition~\ref{p:H},  whether or not $\gamma\in \pi_1(\Gamma,x)$ is simple in the finite rank free group $\pi_1(\Gamma,x)$. If $\gamma\in \pi_1(\Gamma,x)$ is simple in $\pi_1(\Gamma,x)$, we conclude that the graph $(\Gamma,x)$ belongs to $\mathcal G_0(w)$, and $\gamma\in \pi_1(\Gamma,x)$ is not simple in $\pi_1(\Gamma,x)$, we conclude that he graph $(\Gamma,x)$ does not belong to $\mathcal G_0(w)$.
Performing this procedure for each partition of $VC_w$ as a disjoint union of nonempty subsets produces the finite set $\mathcal G_0(w)$.
Proposition~\ref{p:criterion} then implies that $\dsi(g)=\dsi(w)=\min\{\#V\Gamma: (\Gamma,x) \in \mathcal G_0(w)\}$.

The algorithm for computing $\dpr(g)=\dpr(w)$ is similar. We first find all the graphs in $\mathcal G(w)$ as follows.
Enumerate all partitions of $VC_w$ as a disjoint union of nonempty subsets. For each such partition $V_1,\dots V_m$ collapse each $V_i$, $i=1,\dots, m$, to a vertex and then fold the result to get a principal quotient $(\Gamma,x)$ of $C_w$. There is a path $\gamma$ from $x$ to $x$ in $\Gamma$ labeled $w$. Then check, using the algorithm from part (2) of Proposition~\ref{p:H},, whether or not $\gamma\in \pi_1(\Gamma,x)$ is primitive in the free group $\pi_1(\Gamma,x)$. If yes, we conclude that $(\Gamma,x)\in \mathcal G(w)$ and if not, we conclude that $(\Gamma,x)\not\in \mathcal G(w)$. This procedure algorithmically computes the set  $\mathcal G(w)$.

Proposition~\ref{p:criterion} then implies that $\dpr(g)=\dpr(w)=\min\{\#V\Gamma: (\Gamma,x) \in \mathcal G(w)\}$.
Thus part (1) of the theorem is verified.

Part (2) now follows directly from part (1) using the definitions of $\fp(n)$ and $\fs(n)$.

\end{proof}

\begin{rk}
The complexity of the algorithms for computing $\dsi(g)$ and $\dpr(g)$ given in part (1) of Theorem~\ref{t:algorithm} is super-exponential in $n=||g||_A$. 
The reason is that enumerating all principal quotients of the graph $C_w$ requires listing all partitions of the $n$-element set $VC_w$. The \emph{Bell number} $B_n$, which is the number of all partitions  of an $n$-element set, grows roughly as $n^n$. 
\end{rk}

\subsection{Algorithmic computability of $\dfi(g)$}

We now want to give an algorithm for computing $\dfi(g)$. Computationally this algorithm is not nearly as nice as the algorithms for computing $\dsi(g)$ and $\dpr(g)$ described above.

We briefly recall here some definitions and notations related to the Outer space. We refer the reader to \cite{Gui,KL,Vog} for more details. Let $N\ge 2$ be an integer. The \emph{unprojectivized Outer space} $\cvn$ is the set of all  of $F_N$-equivariant isometry classes of $\mathbb R$-trees $T$ such that $T$ is equipped with a free discrete minimal isometric action of $F_N$. The \emph{projectivized Outer space} $\CVN$ consists of the projective classes $[T]$ where $T\in \cvn$. Here for $T\in \cvn$ the projective class $[T]$ of $T$ is the set of all $cT\in \cvn$ where $c\in \mathbb R_{\ge 0}$. Here $cT$ is the same set as $T$, with the same action of $F_N$, but where the metric on $cT$ is the multiple by $c$ of the metric on $T$.

The space $\cvnbar$ is the closure of $\cvn$ in the equivariant Gromov-Hausdorff convergence topology. It is known that $\cvnbar$ consists precisely of all of $F_N$-equivariant isometry classes of $\mathbb R$-trees $T$ such that $T$ is equipped with a free minimal very small isometric action of $F_N$.
The space $\CVNbar$ is the set of all projective classes $[T]$ where $T\in \cvn$ (the projective class $[T]$ for $T\in \cvnbar$ is defined similarly as above, as the set of f all $cT\in \cvnbar$ where $c\in \mathbb R_{\ge 0}$).

The following result provides a useful characterization of filling elements:

\begin{prop}\label{p:fill}
Let $1\ne g\in F_N$. Then $g\in F_N$ is filling if and only if $Stab_{Out(F_N)}([g])$ is finite.
\end{prop}
\begin{proof}
Solie \cite[Lemma 2.42, Lemma 2.44]{Solie2} proves that if $g\in F_N$ is non-filling then $Stab_{Out(F_N)}([g])$ is infinite.  Thus the the ``if" direction of Proposition~\ref{p:fill} holds.

Let us now prove the ``only if" direction. Suppose $Stab_{Out(F_N)}([g])$ is infinite. Choose a basepoint $[T_0]\in \CVN$. 
Since the action of $Out(F_N)$ on $\CVN$ is properly discontinuous and since $\CVNbar$ is compact, it follows that there exist an infinite sequence of distinct elements $\phi_n\in Stab_{Out(F_N)}([g])$ and a point $[T]\in \CVNbar -\CVN$ such that $\lim_{n\to\infty} [T_0]\phi^{n}=[T]$. Then for some sequence of scalars $c_n\geq 0$  with $c_n\to 0$ as $n\to\infty$ we have $\lim_{n\to\infty} c_n T_0\phi_n =T$ in $\cvnbar$.  Since $\phi_n([g])=[g]$, it follows that $||g||_T=\lim_{n\to\infty} c_n||\phi_n(g)||_{T_0}=0$. Then by  Proposition~\ref{prop:filling} the element $g$ is not filling in $F_N$, as required. 
\end{proof}

\begin{prop}\label{p:alg-fill}
Let $F_N=F(A)$, where $N\ge 2$ and where $A=\{a_1,\dots, a_N\}$ is a free basis of $F_N$. Then there exists an algorithm that, given a nontrivial element $g\in F_N$, decides whether or not $g$ is filling in $F_N$.
\end{prop}

\begin{proof}

Let $g\in F_N=F(A)$ be a nontrivial freely reduced word. By a result of McCool \cite{Mc} the group $Stab_{Out(F_N)}([g])$ is finitely generated and, moreover, we can algorithmically compute a finite generating set $Y=\{\psi_1,\ldots, \psi_k\}$ of  $Stab_{Out(F_N)}([g])$.

In view of Proposition~\ref{p:fill} we next need to determine if $H:=\langle Y\rangle \leq Out(F_N)$ is finite. Wang and Zimmermann \cite{WZ} prove that for $N>2$, the maximum order of a finite subgroup of $Out(F_N)$ is $2^N N!$. Also, the word problem for $Out(F_N)$ is solvable (even solvable in polynomial time~\cite{Schl}). Thus we then start building the Cayley graph $Cay(H;Y)$ of $H$ with respect to $Y$. Using solvability of the word problem in $Out(F_N)$, for any finite $k$ we can algorithmically construct the ball $B(k)$ of radius $k$ cantered at identity in $Cay(H;Y)$.  We construct the balls $B(2^N N!)$ and $B(1+2^N N!)$. By the result of Wang and Zimmermann mentioned above, the group $H$ is finite if and only if $B(2^N N!)=B(1+2^N N!)$.

Thus we can algorithmically decide whether or not $Stab_{Out(F_N)}([g])$ is finite, and hence, by Proposition~\ref{p:fill}, whether or not $g$ is filling in $F_N$.

\end{proof}

\begin{thm}\label{t:algorithm1} 
Let $F_N=F(A)$, where $N\ge 2$ and where $A=\{a_1,\dots, a_N\}$ is a free basis of $F_N$. Then:
\begin{enumerate}
\item There exists an algorithm that, given $1\ne g\in F_N$, computes $\dfi(g)$.
\item There exists an algorithm that, for every $n\ge 1$ computes $\ff(n)$
\end{enumerate}
\end{thm}
\begin{proof}
Part (2) follows directly from part (1) and from the definition of $\ff(n)$.
Thus we only need to establish part (1).

Given $g\in F_N$, let $w$ be the cyclically reduced form of $g$. Let $C_w$ and its principle quotients be as in Definitions \ref{d:cw}, \ref{d:pq}. Enumerate all principle quotients of $C_w$ as $\{\Gamma_1, \ldots, \Gamma_k\}$. For each $\Gamma_i$ with $1 \leq i \leq k$, two possibilities arise: \\
{\bf{Case (i)}} (\textit{$\Gamma_i$ is not a finite cover of $R_N$}): In this case, we call $\Gamma_i$ a ``success". In this case we can complete $\Gamma_i$ to a finite cover $\Gamma_i'$ of $R_N$ and now $\pi_1(\Gamma_i)$ is a free factor of $\pi_1(\Gamma_i')$. Hence $w$ is simple in in the subgroup represented by $\Gamma_i'$ i.e. $w$ is not filling in the subgroup represented by $\Gamma_i'$. \\
{\bf{Case (ii)}} (\textit{$\Gamma_i$ is a finite cover of $R_N$}): In this case there is a closed loop $\gamma_i$ in $\Gamma_i$ with label $w$. We then use the algorithm from Proposition~\ref{p:alg-fill} to check whether $\gamma_i$ is filling in $\pi_1(\Gamma_i)$. If $\gamma_i$ is not filling in $\pi_1(\Gamma_i)$, and we call $\Gamma_i$ a success. \\
Finally, observe that $\dfi(g)= \min \{V\Gamma_i\,  | \, \Gamma_i \hbox{ is a ``success"}\}$ where this equality is established in a manner similar to that in Proposition \ref{p:criterion}. Thus part (1) of the theorem is proved.
\end{proof}

\section{Special words and finite covers}

The main goal of this section is to find a suitable sufficient condition implying that a given freely reduced word is filling in a given finite index subgroup of $F_N$ represented by a finite cover of the rose $R_N$. Similarly we find a suitable sufficient condition implying that a given freely reduced word is not simple in a subgroup of $F_N$ represented by a given finite cover of the rose $R_N$. 

These goals are accomplished by constructing  ``simplicity blocking" and  ``filling forcing" words in $F_N$ of controlled length, provided by Proposition~\ref{p:five} and Proposition~\ref{p:v} below.
Since the proofs of these Propositions are somewhat technical, we first illustrate the idea of their proof by obtaining a related simpler statement, given in Lemma~\ref{lem:Barysh} below. The proof of Lemma~\ref{lem:Barysh} is due to 
Yuliy Baryshnikov. We then adapt the idea of this proof to obtain  Proposition~\ref{p:v} and Proposition~\ref{p:five}.

\begin{lem}\label{lem:Barysh}
Let $N\ge 2$. Then there exists a constant $c_0=c_0(N)>0$ with the following property.
Let $(\Gamma,\ast)$ be a connected $d$-fold cover of the $N$-rose $R_N$, where $d\ge 1$. Then there exists a freely reduced word $v=v(\Gamma)$ with $|v|\le c_0d^2$ such that for every vertex $x\in V\Gamma$ the path $p(x,v)$ from $x$ labeled by $v$ in $\Gamma$ passes through every topological edge of $\Gamma$ at least once.
\end{lem}
\begin{proof}
The graph $\Gamma$ is a connected $2N$-regular graph with $d$ vertices and $Nd$ topological edges. We can view $\Gamma$ as a directed graph where the directed edges are labeled by elements of $A$ (and without using $A^{-1}$). Then $\Gamma$ is a connected  directed graph where the in-degree of every vertex is equal to $N$, which is also equal to the out-degree of every vertex.
Hence there exists an Euler circuit in $\Gamma$ beginning and ending at $\ast$ consisting of edges labeled by elements of $A$ that transverses each topological edge exactly once. Let $v_1$ be the label of this Euler circuit. Then $v_1$ is freely reduced and no $a_i^{-1}$ occurs in $v_1$ for $i=1, \ldots, N$. Enumerate the vertices as $V\Gamma= \{x_1, x_2, \ldots, x_d\}$ with $\ast=x_1$. Starting at the vertex $x_2$ follow a path $p_1$ with label $v_1$. Denote the terminal vertex of $p_1$ by $z_1$. Let $p_1'$ be an Euler circuit in $\Gamma$ starting and ending at $z_1$ and consisting only of edges labeled by elements of $A$. Let $v_2$ be the label of this path $p_1'$. Note that since we only consider positively labeled edges, the path $p_2=p_1p_1'$ is reduced and its label $v_1v_2$ is a positive (and hence freely reduced) word over $A$. We now inductively define a positive word $v_{i+1}$ over $A$ given that the positive words $v_1, \ldots, v_i$ where $i \in \{1, \ldots, d-1\}$ have already been defined. Starting at vertex $x_{i+1}$ we follow a path $p_{i}$ with label $v_1\ldots v_i$. Denote the terminal vertex of the path $p_i$ by $z_i$. Let $p_i'$ be an Euler circuit at $z_i$ that transverses every positively labeled edge exactly once. Let $v_{i+1}$ be the label of this path $p_i'$. We define our word $v:=v_1v_2\ldots v_d$. Since following a path with label $v_1\ldots v_i$ at any vertex $v_i$ already passes through every topological edge of $\Gamma$ at least once, so does following a path with label $v$. Since each $|v_i|=Nd$ for $1=1, \ldots, d$, we have that $|v|=Nd^2$. 
\end{proof}

\subsection{Simplicity blocking words and finite covers} 
In the above proof the concatenation argument always produces reduced
edge-paths because we only deal with edges and paths labeled by
positive words over $A$.
By contrast, in proving Proposition~\ref{p:v} simple concatenation
does not always work as it may result in paths that are not reduced.
Also, instead of paths labeled by $v$ passing through every edge of
$\Gamma$, we need to ensure a more complicated condition which implies
that all paths labeled by $v$ in $\Gamma$ pass through a certain
``simplicity-blocking'' path $\alpha(\Gamma,T)$, which is defined
below.

\begin{df}\label{d:delta}
Let $\Gamma$ be a finite connected folded $A$-graph, let $T\subseteq \Gamma$ be a maximal tree in $\Gamma$ and let $S_T$ be the corresponding basis of 
$\pi_1(\Gamma,\ast)$. Let $u=y_1\dots y_n$ be a nontrivial freely reduced word over $S_T^{\pm 1}$.  Thus each $y_i$ corresponds to an edge $e_i\in E(\Gamma-T)$. 
We define a reduced path $\delta(u)$ in $\Gamma$ as 
\[
\delta(u):=[\ast, o(e_1)]_T e_1 [t(e_1),o(e_2)]_T e_2\dots \dots e_n [t(e_n),\ast]_T.
\]
\end{df}
Note that if $d=\#V\Gamma$ then $T$ has $\le d-1$ topological edges and hence $|\delta(u)|\le n+(n+1)(d-1)=nd+d-1=d(n+1) -1$. 

\begin{df} \label{d:alpha}
Let $(\Gamma,\ast)$ be a finite folded core graph with a base-vertex $\ast$.  Let $T\subseteq \Gamma$ be a maximal subtree in $\Gamma$. Let $S_T=\{b_1,\dots, b_r\}$ be the basis of $\pi_1(\Gamma,\ast)$ dual to $T$.

Define a reduced edge-path $\alpha(\Gamma,T)$ from $\ast$ to $\ast$ in $\Gamma$ as

\[
\alpha(\Gamma,T):=\delta(b_r^2b_1^2\ldots b_r^2).
\]

\end{df}

\begin{rk}\label{r:alpha}
Note that the path $\alpha(\Gamma,T)$ is reduced and represents the element  $b_r^2b_1^2\ldots b_r^2$ in $\pi_1(\Gamma,\ast)$. The following proposition demonstrates the ``simplicity-blocking''
property of $\alpha(\Gamma,T)$.  The word  $b_r^2b_1^2\ldots b_r^2$ has length $2r+2$ and hence $|\alpha(\Gamma,T)|\le d(2r+3)-1$ where $d=\#V\Gamma$. 
In particular, if $\Gamma$ is a $d$-fold cover of the rose $R_N$, then $r=d(N-1)+1$ and 
\[
|\alpha(\Gamma,T)|\le d(2d(N-1)+3)-1\le 2d^2(N-1)+4d.
\]
\end{rk}

\begin{prop}
Let $\Gamma$ be as in Definition \ref{d:alpha} with $T$ a maximal tree in $\Gamma$. Let $S_T$ and $\alpha(\Gamma,T)$ be as before. Let $\gamma \in \pi_1(\Gamma,\ast)$ be such that $\gamma$ is represented by a cyclically reduced circuit in $\Gamma$ containing $\alpha(\Gamma,T)$ as a subpath. Then $\gamma$ is not simple in $\pi_1(\Gamma,\ast)$.
\end{prop}
\begin{proof}
We first use Proposition \ref{p:rw-basis} to rewrite $\gamma$ as a
cyclically reduced word $w$ in $S_T=\{b_1,\dots,b_r\}$. Then  the
occurrence of $\alpha(\Gamma,T)$ in  $\gamma$ produces an occurrence
of the reduced word $b_r^2b_1^2 \ldots b_r^2$ in $w$. Hence, by
Corollary~\ref{c:np}, in this case $\gamma$ is not simple in $F(b_1,\dots,b_r)=\pi_1(\Gamma,\ast)$.

\end{proof}

Note that Definition~\ref{d:Core-Graph} of a core graph implies that if $\Gamma$ is a finite connected core graph, then $\Gamma$ does not have any degree-1 vertices.

\begin{lem} \label{l:rp}
Let $\Gamma$ be a finite connected core graph with $d$ vertices. Suppose that $\pi_1(\Gamma)$ has rank $\geq 2$. Then for any any two edges $e_1, e_2 \in E(\Gamma)$, there exists a reduced path $p(e_1, e_2)$ starting at $e_1$, ending at $e_2$, and with $|p(e_1,e_2)| \leq 3d$. 
\end{lem}

\begin{proof}
Pick a graph $\Gamma' \subseteq \Gamma$ such that $\Gamma'$ is a finite, connected, core graph with $\pi_1(\Gamma')$ of rank 2 and $e_1,e_2 \in E\Gamma'$. Then there are precisely three possibilities for $\Gamma'$. It can be the wedge of two circles, or a theta-graph (a circle with a line segment joining two points on the circle), or a barbell graph (two circles attached to two ends of a line segment). We will show that the result holds for the graph $\Gamma'$, and hence holds for our graph $\Gamma$. Our proof is essentially going to be a proof by picture for each of these three cases. In Figure \ref{rp}, green edges (or arrows) indicate $e_1$ and blue edges (or arrows) indicate $e_2$. We indicate the path $p(e_1,e_2)$ in red with the $\bullet$ representing the starting point of $p(e_1,e_2)$ and the $\rightarrow$ representing the direction. The path $p(e_1,e_2)$ starts at $o(e_1)$ and ends at $t(v_1)$. We call a ``cusp" any vertex of $\Gamma'$ of degree $\ge 3$ in $\Gamma'$. The idea behind finding this path $p(e_1,e_2)$ is always to travel along $e_1$ to the nearest cusp. Then if one is required to go back on the same path one has already been on to get to $e_2$, one instead travels along a disjoint loop at the cusp. Now one can go back to $e_2$ and the path $p(e_1,e_2)$ will be reduced. If after traveling from $e_1$ to the cusp one can get to $e_2$ without compromising the fact that the path $p(e_1,e_2)$ is reduced, then one simply goes to $e_2$ and the path $p(e_1,e_2)$ so obtained is reduced. From Figure \ref{rp} we see that the result holds. 

\begin{figure}

\includegraphics[trim=5cm 10cm 5cm 7cm, clip=true, totalheight=0.3\textheight, angle=0]{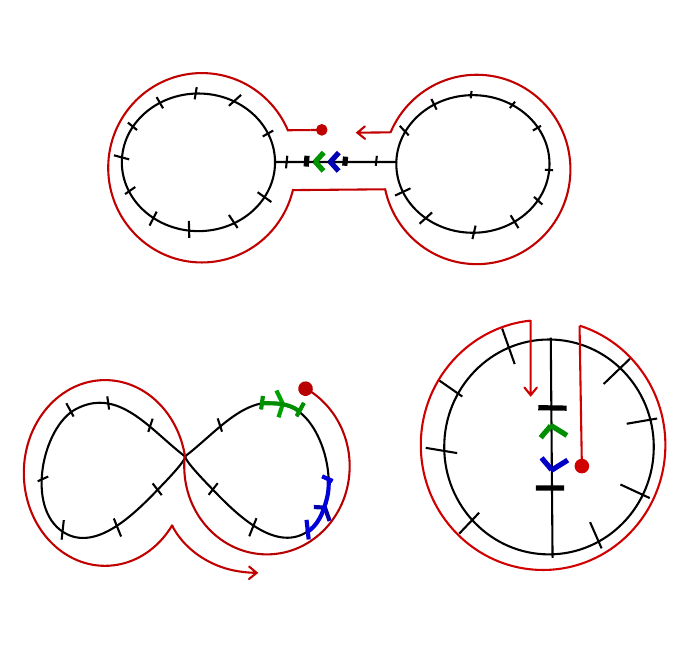}
\caption{Proof by picture for Lemma \ref{l:rp} \label{rp}}

\end{figure}

\end{proof}

The following fact plays a key role in the proof of Theorem~\ref{thm:B}:

\begin{prop}\label{p:v}
Let $N\ge 2$. Then there exists a constant $c_0=c_0(N)>0$ with the following property.
Let $(\Gamma,\ast)$ be a connected $d$-fold cover of the $N$-rose $R_N$, where $d\ge 1$ and let $T\subseteq \Gamma$ be a maximal subtree of $\Gamma$. Then there exists a freely reduced word $v=v(\Gamma,T)$ with $|v|\le c_0d^3$ such that for every vertex $x\in V\Gamma$ the path $p(x,v)$ from $x$ labeled by $v$ in $\Gamma$ contains $\alpha(\Gamma,T)$ as a subpath.
\end{prop} 
\begin{proof}
Let us begin by enumerating the vertices of $V\Gamma= x_1, x_2, \ldots, x_d$. Let $H \leq F_N$ be the subgroup of index $d$ that is represented by $(\Gamma,\ast)$. Let $T$ be a maximal tree in $(\Gamma,\ast)$ and let $S_T=\{b_1, b_2, \ldots, b_r\}$ be the corresponding basis of $\pi_1(\Gamma,\ast)$. 
By Remark~\ref{r:alpha}, we have $|\alpha(\Gamma, T)| \leq  2d^2(N-1) + 4d$.

Let $e$ be the first edge of the path $\alpha(\Gamma,T)$. Starting at the vertex $x_1 \in V\Gamma$, there exists a unique path $[x_1, \ast]_T$ of length $\leq d-1$ with terminal edge $e_1$ (say). Lemma \ref{l:rp} then gives us a reduced path $p(e_1,e)=e_1p'e$ of length $\leq 3d$. Let the word $v_1$ be the label of the path $p_1=[x_1, \ast]_Tp'\alpha(\Gamma, T)$. Note that $|v_1|=|p_1| \leq 2d^2(N-1)+8d-3$. 

 Starting at the vertex $x_2$ we follow a path $p_1'$ that has label $v_1$. Let $e_2$ be the terminal edge of the path $p_1'$. Then from Lemma \ref{l:rp}, the path $p(e_2,e)=e_2p_1''e$ is reduced with $|p(e_2,e)| \leq 3d$, and hence $|p_1''|\leq 3d-2$. Let the word $v_2$ be the label of the path $p_2=p_1''\alpha(\Gamma,T)$. Now the path $p_1'p_2=p_1'p_1''\alpha(\Gamma,T)$ is reduced. Notice that $|v_2|= |p_2| \leq 2d^2(N-1)+7d-1$. We now define inductively a sequence of words and paths as follows: Suppose we have already defined our words $v_1, v_2, \ldots, v_{i-1}$ which are respectively the labels of reduced paths $p_1, \ldots, p_{i-1}$. Starting at vertex $x_i$ we follow the path $p_{i-1}'$ labeled by the word $v_1v_2 \ldots v_{i-1}$. Let $e_{i}$ be the terminal edge of the path $p_{i-1}'$. Then the path $p(e_{i}, e)=e_{i}p_{i-1}''e$ is reduced with $|p_{i-1}''| \leq 3d-2$. Let the word $v_i$ be the label of the reduced path $p_i= p_{i-1}''\alpha(\Gamma,T)$. Now the path $p_{i-1}'p_i=p_{i-1}'p_{i-1}''\alpha(\Gamma,T)$ is reduced. Let the word $v:= v_1v_2 \ldots v_d$. Then notice that at any vertex $x_i$ with $1 \leq i \leq d$, the path $p_{i-1}'p_i$ is a reduced path labeled by $v_1 \ldots v_i$ that already contains the subpath $\alpha(\Gamma,T)$. Thus for $i=1,\dots, d$ the path starting at $x_i$ labeled by the word $v_1\ldots v_d$ also contains the subpath $\alpha(\Gamma,T)$. Since for all $2 \leq i \leq d$, $|v_i| \leq 2 d^2(N-1) + 7d-1$, we have that $|v| \leq 2 d^3(N-1) + 7d^2-2 \leq (2N+5)d^3$. Thus with $c_0=2N+5$, we are done.

\end{proof} 

If $\Gamma$ is a finite connected graph and $T\subseteq \Gamma$ is a maximal subtree, then, following the conventions of Bass-Serre theory, we denote 
\[
\pi_1(\Gamma,T):=\langle E\Gamma | e\bar e=1 \text{ for all } e\in E\Gamma,   e=1 \text{ for all } e\in ET\rangle. 
\]
If $\Gamma$ is equipped with an orientation, then $\pi_1(\Gamma,T)$ is canonically isomorphic to the free group $F(E_+\Gamma-E_+T)$.
Note also that $\pi_1(\Gamma,T)$ is isomorphic to the fundamental group of the quotient space $\Gamma/T$ (where $T$ is collapsed to a point).

The freely reduced word  $v=v(\Gamma,T)$ in $F(A)$  can be viewed as a ``simplicity blocking" word for the elements of the fundamental group of a $d$-fold cover $\Gamma$ of $R_N$. 

\begin{cor}\label{cor:v}
Let $N\ge 2$ and let $c_0=c_0(N)>0$ be the constant provided by Proposition~\ref{p:v}. 

Let $d\ge 1$, let $\Gamma$ be a connected $d$-fold cover of the $N$-rose $R_N$ and let $T\subseteq \Gamma$ be a maximal tree in $\Gamma$.  Let $\ast\in V\Gamma$, let $\gamma$ be a reduced edge-path from $\ast$ to $\ast$  in $\Gamma$ and let $\gamma'$ be the cyclically reduced form of the path $\gamma$ (so that the label of $\gamma'$ is a cyclically reduced word in $F(A)$). Suppose that the label of $\gamma'$ contains as a subword the word $v=v(\Gamma,T)$ with $|v|\le c_0d^3$  provided by Proposition~\ref{p:v}. 

Then $\gamma\in \pi_1(\Gamma,\ast)$ does not belong to a proper free factor of $\pi_1(\Gamma,\ast)$.

\end{cor}
\begin{proof}
From definitions $\gamma \in \pi_1(\Gamma,\ast)$. Using the tree $T$ we can obtain a free basis $S_T=\{b_1,\ldots, b_r\}$ of $\pi_1(\Gamma,T)$. Then Proposition \ref{p:rw-basis} tells us how to rewrite $\gamma$ in terms of the basis $S_T$, both as freely reduced word and as a cyclically reduced word. Let $\alpha(\Gamma,T)$ be as before. Then for the label of $\gamma'$ to contain the word $v$, we must have that the cyclically reduced form of $\gamma'$ in terms of $S_T$ contains $b_r^2b_1^2\ldots b_r^2$ as a subword. Now from Corollary \ref{c:np} we know that $\gamma'$ is not simple in $\pi_1(\Gamma,T)$. Finally from Lemma \ref{l:aut} $\gamma$ is not simple in $\pi_i(\Gamma,T)$, that is, $\gamma\in \pi_1(\Gamma,\ast)$ does not belong to a proper free factor of $\pi_1(\Gamma,\ast)$.
 
\end{proof}

\subsection{Filling forcing words and finite covers}
To proceed further we will once again adapt the idea of proof of Lemma \ref{lem:Barysh} to produce  a ``filling-forcing" path $\beta(\Gamma, T)$ of controlled length.

\begin{conv}
If $F(B)$ is a free group with $|B|=r\ge 2$, then the total number of freely reduced words of length 3 in $F(B)$ are $L=2r(2r-1)^2$. Let $\{u_1, \ldots, u_L\}$ be the set of all freely reduced words of length 3 in $B^{\pm1}$. Define a freely reduced word $u_B:= u_1y_1u_2y_2u_3y_3\ldots y_{L-1}u_L$ where each $y_i$ is either the empty word, or $y_i \in B^{\pm 1}$. Namely, whenever the concatenation $u_ju_{j+1}$ is reduced to begin with, we define $y_j$ to be the empty word. If this concatenation is not reduced, then we can always choose $y_j\in B^{\pm 1}$ so that $u_jy_ju_{j+1}$ is reduced in $F(B)$.  Note that  $|u|_{B}\leq 3L+L-1 = 4L-1$. 
\end{conv}

We now define the path $\beta(\Gamma, T)$ as follows:
\begin{df} \label{d:beta}
Let $(\Gamma,\ast)$ be a finite connected folded core graph with a base-vertex $\ast$.  Let $T\subseteq \Gamma$ be a maximal subtree in $\Gamma$ with $E_+(\Gamma-T)=\{e_1,\dots,e_r\}$, and let $S_T=\{b_1,\dots, b_r\}$ be the basis of $\pi_1(\Gamma,\ast)$ dual to $T$.
We put 
\[
\beta(\Gamma,T):=\delta(u_{S_T}).
 \]
Thus $\beta(\Gamma,T)$ is a reduced edge-path from $\ast$ to $\ast$ in $\Gamma$ representing the element $u_{S_T}$ in $\pi_1(\Gamma,\ast)$.
Recall that $u_{S_T}$ has length $4L-1=8r(2r-1)^2-1$. Therefore
\[
|\beta(\Gamma,T)|\le 8rd(2r-1)^2 -1
\]
where $d=\#V\Gamma$.
In particular, if $\Gamma$ is a $d$-fold cover of $R_N$ then $r=d(N-1)+1$ and
\[
|\beta(\Gamma,T)|\le 8d (d(N-1)+1)( 2d(N-1)+1)^2-1\le 500d^4N^3.
\]
\end{df}

The following proposition demonstrates the a ``filling-forcing" property of the path $\beta(\Gamma, T)$

\begin{prop}\label{p:T}
Let $\Gamma$ be as in Definition \ref{d:beta} with $T$ a maximal tree. Let $S_T$ and $\beta(\Gamma,T)$ be as before. Let $\gamma \in \pi_1(\Gamma,\ast)$ be such that $\gamma$ is represented by a cyclically reduced circuit in $\Gamma$ containing $\beta(\Gamma,T)$ as a subpath. Then $\gamma$ is filling in $\pi_1(\Gamma,\ast)$.
\end{prop}
\begin{proof}
We first use Proposition \ref{p:rw-basis} to rewrite $\gamma$ as a
cyclically reduced word $w$ in $S_T=\{b_1,\dots,b_r\}$. Then  the
occurrence of $\beta(\Gamma,T)$ in  $\gamma$ produces an occurrence
of the reduced word $u_1y_1u_2y_2u_3y_3\ldots y_{l-1}u_l$ in $w$. Since every reduced word of length 3 now occurs in $w$, by
Proposition~\ref{p:ChM} $\gamma$ is filling in $F(b_1,\dots,b_r)=\pi_1(\Gamma,\ast)$.

\end{proof}

We are now in a position to prove a key proposition that is used in the proofs of  Theorem~\ref{thm:A} and Theorem~\ref{thm:B}:

\begin{prop}\label{p:five}
Let $N\ge 2$. Then there exists a constant $c_1=c_1(N)>0$ with the following property.
Let $(\Gamma,\ast)$ be a connected $d$-fold cover of the $N$-rose $R_N$, where $d\ge 1$ and let $T\subseteq \Gamma$ be a maximal subtree of $\Gamma$. Then there exists a freely reduced word $w=w(\Gamma,T)$ with $|w|\le c_1d^5$ such that for every vertex $x\in V\Gamma$ the path $p(x,w)$ from $x$ labeled by $w$ in $\Gamma$ contains $\beta(\Gamma,T)$ as a subpath.
\end{prop} 
\begin{proof}
Let us begin by enumerating the vertices of $V\Gamma= \{x_1, x_2, \ldots, x_d\}$. Let $H \leq F_N$ be the subgroup of index $d$ that is represented by $(\Gamma,\ast)$.
We have seen above that $|\beta(\Gamma,T)|\le 500d^4N^3$.

Let $e$ be the first edge of the path $\beta(\Gamma,T)$. Starting at the vertex $x_1 \in V\Gamma$, there exists a unique path $[x_1, \ast]_T$ of length $\leq d-1$ with terminal edge $e_1$ (say). Lemma \ref{l:rp} then gives us a reduced path $p(e_1,e)=e_1p'e$ of length $\leq 3d$. Let the word $w_1$ be the label of the path $p_1=[x_1, \ast]_Tp'\beta(\Gamma, T)$. Note that $|w_1|\le 500d^4N^3 +3d$.

 Starting at the vertex $x_2$ we follow a path $p_1'$ that has label $w_1$. Let $e_2$ be the terminal edge of the path $p_1'$. Then from Lemma \ref{l:rp}, there is a reduced path $p(e_2,e)=e_2p_1''e$  with $|p(e_2,e)| \leq 3d$ and $|p_1''|\leq 3d-2$. Let the word $w_2$ be the label of the path $p_2=p_1''\beta(\Gamma,T)$.  Thus $|w_2|=|p_2| \le 500d^4N^3 +3d$.
 
 Now the path $p_1'p_2=p_1'p_1''\beta(\Gamma,T)$ is reduced, starts at $x_2$, ends in $\beta(\Gamma, T)$, has label $w_1w_2$ and has length
\[
|p_1'p_2|=|w_1w_2|\le 2(500d^4N^3 +3d).
\]
 
We proceed inductively as follows.  For $2\le i\le d$ suppose that we have already constructed freely reduced words $w_1, \dots, w_{i-1}\in F_N=F(A)$ of length $|w_j|\le   500d^4N^3 +3d$ such that the word $w_1\dots w_{i-1}$ is freely reduced and such that reading $w_1\dots w_{i-1}$ from the vertex $x_{i-1}$ gives a reduced path in $\Gamma$  ending in $\beta(\Gamma,T)$.
 
 Starting at vertex $x_i$ we follow the path $p_{i-1}'$ labeled by the word $w_1w_2 \ldots w_{i-1}$. Let $e_{i}$ be the terminal edge of the path $p_{i-1}'$. Then the path $p(e_{i}, e)=e_{i}p_{i-1}''e$ is reduced with $|p_{i-1}''| \leq 3d-2$. Let the word $w_i$ be the label of the reduced path $p_i= p_{i-1}''\beta(\Gamma,T)$.  We again have $|w_i|\le   500d^4N^3 +3d$. 
Now the path $p_{i-1}'p_i=p_{i-1}'p_{i-1}''\beta(\Gamma,T)$ is reduced, starts with $x_i$ and ends in $\beta(\Gamma,T)$, completing the inductive step.

Finally let  $w:= w_1w_2 \ldots w_d$. Then $w$ is freely reduced, has $|w|\le 500d^5N^3 +3d^2\le 1000N^3 d^5$.
By construction $w$ has the property that for $i=1,\dots, d$ reading $w$ from $x_i$ gives a path in $\Gamma$  containing $\beta(\Gamma,T)$ as a subpath. We put $w(\Gamma,T):=w$ and $c_1=1000N^3$. The conclusion of the proposition now holds.
\end{proof} 

The freely reduced word  $w=w(\Gamma,T)$ in $F(A)$  can be viewed as a ``filling forcing" word for the elements of the fundamental group of a $d$-fold cover $\Gamma$ of $R_N$.

 \section{A lower bound for the non-filling index function}

\begin{rk}
Let $F_N = F(a_1, \ldots, a_N)$ be free of rank $N\ge 2$, as before. 
It is well-known (see, for example, \cite{LS}) that for an integer $d\ge 1$ there are $\leq (d!)^N$ subgroups of index $d$ in $F_N$.
Indeed, every subgroup of index $d$ in $F_N$ can be uniquely represented by a finite connected folded $2N$-regular $A$-graph on vertices $1,\dots, d$, where $1$ is viewed as a base-vertex. 
Every such graph $\Gamma$ is uniquely specified by choosing an ordered $N$-tuple of permutations in $S_d$. Indeed, if $\sigma_1,\dots, \sigma_N\in S_d$, we construct $\Gamma$ with $V\Gamma=\{1,\dots, d\}$ by putting an edge from $j$ to $\sigma_i(j)$ labeled by $a_i$ for $1\le i\le N$,  and $1\le j\le d$. 

Thus indeed $F_N$ has $\leq (d!)^N$ subgroups of index $d$ and it has $\leq d(d!)^N$ subgroups of index $\le d$.  
\end{rk}

\begin{thm} \label{t:newmain}
Let $N\geq2$ and let $F_N = F(A)$ where $A = a_1, \ldots, a_N$. Then there exists a constant $c > 0$ and an integer $M\ge 1$ such that 
for all $n \geq M$ we have
\[
\fp(n) \geq \fs(n) \geq \ff(n) \ge  c \frac{\log n}{ \log \log n}.
\]
\end{thm}
\begin{proof}

Let $d\ge 1$ be an integer. Denote  $m(d)= m := d(d)!^N$. Enumerate all the subgroups of $F_N$ of index $\le d$ as $H_1,\dots, H_m$ (we do allow repetitions in this list since the actual number of such distinct subgroups is $<m(d)$. Let $\Gamma_1,\dots, \Gamma_m$ be the base-pointed finite covers of the rose $R_N$ representing the subgroups $H_1,\dots, H_m$.

For $i=1,\dots, m$ let $w_i\in F(A)$ be the freely reduced ``filling forcing" word with $|w_i|\le c_1d^5$ corresponding to $\Gamma_i$ as provided by Proposition~\ref{p:five}. 
We can now construct a freely reduced  and cyclically reduced word
\[
z_d:= w_1u_1w_2u_2\ldots u_{m-1}w_mu_m
\] 
where each $u_i$ is either the empty word or $u_i\in\{a_1, \ldots, a_N\}^{\pm 1}$. 
Then 
\[
||z_d||\le c_1 md^5=c_1 d^6 (d!)^N.
\]

We claim that $\dfi(z_d)> d$. Indeed, suppose not, that is suppose that $\dfi(z_d)\le d$. Then there exists $1\le i\le m$ such that $z_d\in H_i$ and that $z_d$ is a non-filling element of $H_i=\pi_1(\Gamma_i,\ast)$. Let $\gamma$ be the path in $\Gamma_i$ from $\ast$ to $\ast$ labeled by $z_d$. By Proposition~\ref{p:five} the fact that $z_d$ is cyclically reduced and contains $w_i$ as subword implies that $\gamma$ contains the path $\beta(\Gamma_i,T)$ as a subword. Hence, by Proposition~\ref{p:T}, $\gamma$ is a filling element in  $\pi_1(\Gamma_i,\ast)$, yielding a contradiction. Thus indeed $\dfi(z_d)> d$.

Now for $d\ge 1$ let $n_d:=c_1 d^6 (d!)^N$.  We also put $n_0=1$. Then for every integer $d\ge 0$ we have $\ff(n_d)>d$.
By Stirling's formula, there is $C>0$ such that for all sufficiently large $d\ge 1$ we have 
\[
d\ge C \frac{\log n_d}{\log \log n_d} \tag{$\dag$}
\]
Similarly, using a standard calculus argument we see that for all sufficiently large $d$ we have  
\[
\displaystyle \frac{\log(n_{d-1})}{\log\log(n_{d-1})} \ge \frac{1}{2} \frac{\log(n_d)}{\log\log(n_d)}. \tag{$\ddag$}
\]
Let $d_0\ge 2$ be such that for  all $d\ge d_0$ the inequalities $(\dag)$ and $(\ddag)$ hold and that the function  the function $\displaystyle \frac{\log x}{\log\log x}$
is monotone increasing on the interval $[n_{d_0-1},\infty)$.

Now let $n\ge n_{d_0+1}$ be an arbitrary integer.  There exists a unique $d\ge 0$ such that $n_{d-1}< n \le n_d$.
Since $\ff(n)$ is a non-decreasing function, we get that $\ff(n) \geq \ff(n_{d-1}) > d-1$ and $d-1\ge d_0$. 

Then
\begin{gather*}
\ff(n)\ge  \ff(n_{d-1}) > d-1 \ge C \frac{\log (n_{d-1})}{\log \log (n_{d-1}) }\ge \frac{C}{2} \frac{\log(n_d)}{\log\log(n_d)} \ge  \frac{C}{2} \frac{\log n}{\log \log n},
\end{gather*}
and the conclusion of the theorem follows.
\end{proof}

\color{black}

\section{Non-backtracking simple random walk on $F_N$}\label{sect:ttdrw}

\begin{conv}\label{conv:O}
In this paper we use the standard big-O and big-$\Theta$ conventions. For functions $f,g: \mathbb N\to \mathbb R$ we write $f=O(g)$ (or sometimes $f(n)=O(g(n))$) if there exist an integer $n_0\ge 1$ and a constant $C>0$ such that for all integers $n\ge n_0$ we have $|f(n)|\le C |g(n)|$. For such $f,g$ we write $f=\Theta(g)$ if $f=O(g)$ and $g=O(f)$. 
In particular, if $f(n)=O(g(n))$ and $\lim_{n\to\infty} g(n)=0$ then $\lim_{n\to\infty} f(n)=0$.
\end{conv}

Recall that we set for the free group $F_N=F(A)=F(a_1,\dots, a_N)$ (where $N\ge 2$) a distinguished free basis $A=\{a_1,\dots, a_N\}$.
Put $\Upsilon=A\cup A^{-1}$.

\begin{df}\label{d:X}
We consider the following finite-state Markov chain $\mathcal X$.  The set of states for $\mathcal X$ is $\Upsilon$. For $x,y\in \Upsilon$, the transition probability $P_{x,y}$ from $x$ to $y$ is defined as:
\[
P_{x,y}:=\begin{cases}\frac{1}{2N-1}, \quad \text{ if } y\ne x^{-1}\\
0, \quad \text{if } y=x^{-1}\end{cases}.
\]
\end{df}

Let $M$ be the transition matrix of $\mathcal X$. That is, $M$ is a $2N\times 2N$ matrix with columns and rows indexed by $\Upsilon$ where for $x,y\in \Upsilon$ the entry $m_{x,y}$ in $M$ is equal to $1$ if $y\ne x^{-1}$ and is equal to $0$ if $y=x^{-1}$.

We summarize the following elementary properties of $\mathcal X$, which easily follow from the definitions:

\begin{lem}
Let $N\ge 2$ and $\mathcal X$ be as in Definition~\ref{d:X}. Then:
\begin{enumerate}
\item $\mathcal X$ is an irreducible aperiodic finite-state Markov chain.
\item The uniform probability distribution $\mu_1$ on $\Upsilon$ is stationary for $\mathcal X$.
\item The matrix $M$ is an irreducible aperiodic nonnegative matrix with the Perron-Frobenius eigenvalue $\lambda=2N-1$.
\end{enumerate}
\end{lem}
\begin{proof}
For any $x,y\in \Upsilon$ there exists $z\in \Upsilon$ such that $xzy$ is a freely reduced word. Hence $P_{x,z}P_{z,y}>0$, which means that $\mathcal X$ is an irreducible Markov chain. The fact that for every $x\in \Upsilon$, we have $P_{x,x}>0$ implies that $\mathcal X$ is aperiodic. Thus (1) is verified.

Part (2) easily follows from the definition of $\mathcal X$ by direct verification.

Part (1) implies that $M$ is an irreducible aperiodic nonnegative matrix. Therefore, by the basic Perron-Frobenius theory, the spectral radius $\lambda:=\max\{|\lambda_\ast| : \lambda_\ast\in \mathbb C \text{ is an eigenvalue of } M\}$ is a positive real number which is itself an eigenvalue of $M$ called the Perron-Frobenius eigenvalue of $M$. It is also known that $\lambda$ admits an eigenvector with strictly positive coordinates, and that any other eigenvalue of $M$ admitting such an eigenvector is equal to $\lambda$. It is easy to see from the definition of $M$ that for the vector $v$ with all entries equal to $1$ we have $Mv=(2N-1)v$. Therefore $\lambda=2N-1$, as claimed.
\end{proof}

Let $\Omega=\Upsilon^\N=\{\omega=x_1,x_2,\dots | x_i\in \Upsilon\}$. We put the discrete topology on $\Upsilon$ and the product topology on $\Omega$ so that $\Omega$ becomes a compact Hausdorff space.  For every finite word $\sigma\in \Upsilon^\ast$ the \emph{cylinder} $Cyl(\sigma)\subseteq \Omega$ consists of all sequences $\omega\in \Omega$ with $\sigma$ as the initial segment. For each $\sigma\in \Upsilon^\ast$ the set $Cyl(\sigma)$ is compact and open in $\Omega$ and the sets $\{Cyl(\sigma)| \sigma\in \Upsilon^\ast\}$ provide a basis for the product topology on $\Omega$.

By using the uniform distribution $\mu_1$ on $\Upsilon$ as the initial distribution for $\mathcal X$, the Markov chain $\mathcal X$ defines a Borel probability measure $\mu$ on $\Omega$ via the standard convolution formula: 

For $\sigma=x_1\dots x_n\in \Upsilon^\ast$, 
\[
\mu(Cyl(\sigma))=\mu_1(x_1)P_{x_1,x_2}\dots P_{x_{n-1},x_n}.
\] 

Note that the support of $\mu$ is exactly $\partial F_N$, that is, the set of all semi-infinite freely reduced words $\omega=x_1,x_2,\dots $ over $\Upsilon$.

\begin{conv}
For $\sigma\in \Upsilon^\ast$ (where $\Upsilon^\ast$ is the set of all words over the alphabet $\Upsilon$) we denote $\mu(\sigma):=\mu(Cyl(\sigma))$. Also, for the remainder of this section we denote $\lambda:=2N-1$.
\end{conv}

The following is a direct corollary of the definitions:

\begin{lem}
Let $\sigma=x_1\dots x_n\in \Upsilon^{\ast}$, where $n\ge 1$.
Then
\[
\mu(\sigma)=\begin{cases} \frac{1}{2N(2N-1)^{n-1}}, \quad \text{ if $\sigma$ is freely reduced},\\
0, \quad \text{ if $\sigma$ is not freely reduced.}
 \end{cases}
\]
\end{lem}

\begin{notation}
Let $v,w\in \Upsilon^\ast$. We denote by $\langle v,w\rangle$ the number of times the word $v$ occurs as a subword of $w$.

For $n\ge 1$ let $S(n)$ be the set of all freely reduced words of length $n$ in $\Upsilon^\ast$ (so that $\#(S(n))=2N(2N-1)^{n-1}=\frac{2N}{2N-1}\lambda^n$), and let $\mu_n$ be the uniform probability distribution on $S(n)$.
\end{notation}

The following statement is a special case, when applied to $\mathcal X$, of Proposition~3.13 in~\cite{CM} (which in turn is based on the proof of the main result of Dinwoodie~\cite{Din}).

\begin{prop}\label{p:CM}
Let $\epsilon>0$ and $0<\ell <1$. Then there exist constants $C_1>1$ and $C_2>0$, depending on $\epsilon$ and $\ell$, with the following property. Let  $n\ge 1$ and $\sigma\in \Upsilon^\ast$ be a freely reduced word be such that $|\sigma|=\ell \log_\lambda n=\ell \log n/\log \lambda$. Then for $w_n\in S(n)$ we have
\[
1-P_{\mu_n}(\left|\langle \sigma, w_n\rangle -n \mu(\sigma)\right|<n^{\epsilon +(1-\ell)/2}) =O(C_1^{-n^{C_2}} ), 
\]
and therefore, since $\lambda=2N-1$ and $\mu(\sigma)=\frac{2N-1}{2N} \lambda^{-|\sigma|}=\frac{2N-1}{2N} n^{-\ell}$, 
\[
1-P_{\mu_n}(\left|\langle \sigma, w_n\rangle -\frac{2N-1}{2N} n^{1-\ell}\right|<n^{\epsilon +(1-\ell)/2}) =O(C_1^{-n^{C_2}} ), 
\]

\end{prop}

\begin{cor}\label{c:CM}
Let $\epsilon>0$ and $0<\ell <1$. Let constants $C_1=C_1(\epsilon, \ell)>1$ and $C_2=C_2(\epsilon, \ell)>0$ be the constants provided by Proposition~\ref{p:CM}.

\begin{enumerate}
\item Let $n\ge 1$ and let $E_n\subseteq S(n)$ consist of those $w_n\in S(n)$ such that for every  freely reduced $\sigma\in \Upsilon^\ast$ with $|\sigma|=\ell \log_\lambda n=\ell \log n/\log \lambda$
we have
\[
\left|\langle \sigma, w_n\rangle -\frac{2N-1}{2N} n^{1-\ell} \right|<n^{\epsilon +(1-\ell)/2},
\]
Then \[1-P_{\mu_n}(w_n\in E_n)=O\left( n^\ell C_1^{-n^{C_2}}\right).\]
\item Suppose that $\epsilon>0, 0<\ell<1$ are chosen so that $\ell<1-2\epsilon$, and thus $1-\ell>\epsilon +(1-\ell)/2$.
Let $H_n\subseteq S(n)$ consist of all $w_n\in S(n)$ such that for every freely reduced $\sigma$ with $|\sigma|=\ell \log_\lambda n$ we have
\[\langle \sigma, w_n\rangle \ge \frac{2N-1}{4N}n^{1-\ell}.\]  

%Let $n_0\ge 1$ be such that for all $n\ge n_0$ we have $\frac{2N-1}{4N}n^{1-\ell}\ge n^{\epsilon +(1-\ell)/2}$.

Then for  $n\ge n_0$ we have
\[
1-P_{\mu_n}(w_n\in H_n) =O(n^\ell C_1^{-n^{C_2}} ).
\]

\end{enumerate}
\end{cor}
\begin{proof}
For every freely reduced $\sigma$ with $|\sigma|=\ell \log_\lambda n$ let $E'_{n,\sigma}$ consist of all $w_n\in S(n)$ such that $\left|\langle \sigma, w_n\rangle -n \mu(\sigma)\right|\ge n^{\epsilon +(1-\ell)/2}$.
Thus, by Proposition~\ref{p:CM}, for every such $\sigma$ we have $P_{\mu_n}(E'_{n,\sigma})= O(C_1^{-n^{C_2}} )$.

Suppose $w_n\not\in E_n$. Then there exists freely reduced $\sigma\in \Upsilon^\ast$ with $|\sigma|=\ell \log_\lambda n$ such that $w_n\in E'_{n,\sigma}$. Since there are $O(n^\ell)$ freely reduced words $\sigma$ with $|\sigma|=\ell \log_\lambda n$, it follows that $P_{\mu_n}(S(n)\setminus E_n)=O\left( n^\ell C_1^{-n^{C_2}}\right)$. Hence $1-P_{\mu_n}(E_n)=O\left( n^\ell C_1^{-n^{C_2}}\right)$, as required, and part (1) of Corollary~\ref{c:CM} is verified.

Part (2) now directly follows from part (1).
\end{proof}

\begin{notation}
For a freely reduced word $w\in \Upsilon^\ast$ let $\iota(w)$ be the
maximal initial segment of $w$ such that $(\iota(w))^{-1}$ is a
terminal segment of $w$. Let $\tilde w$ be the word obtained by
removing the initial and terminal segments of $w$ of length
$|\iota(w)|$. Thus $\tilde w$ is the cyclically reduced form of $w$.
\end{notation}

The following facts are well-known and easy to check by a direct
counting argument; see \cite{AO} for details:

\begin{lem}\label{lem:cr} The following hold:
\begin{enumerate}
\item For every $0<\epsilon_0 <1$ there exists $C_0>1$ such that for $w_n\in S(n)$
\[
1-P_{\mu_n} (|\iota(w_n)|\le \epsilon_0 n) = O(C_0^{-n}). 
\]
\item There is $C>1$ such that for $w_n\in S(n)$
\[
1-P_{\mu_n}(w_n \text{ is not a proper power in } F_N)=O(C^{-n}). 
\]
\end{enumerate}
\end{lem}

\section{Bounding below the simplicity  and the non-filling index for random elements}

Recall that for a nontrivial element $g\in F_N$ we denote by $\dsi(g)$ the smallest $d\ge 1$ such that there exists a subgroup $H\le F_N$ with $[F_N:H]\le d$ such that $g\in H$ and, moreover, that $g$ belongs to a proper free factor of $H$. Similarly, for $g\in F_N-\{1\}$ we denote by $\dpr(g)$ the smallest $d\ge 1$ such that there exists a subgroup $H\le F_N$ with $[F_N:H]\le d$ such that $g\in H$ and, moreover, that $g$ is primitive in $H$.  As we have seen, for every $g\in F_N-\{1\}$ we have $\dsi(g)\le \dpr(g)\le ||g||_A$, where $A=\{a_1,\dots, a_n\}$ is a free basis of $F_N$. Recall that for $n\ge 1$ we denote by $\mu_n$  the uniform probability distribution on the sphere $S(n)\subseteq F(A)=F_N$.

For the remainder of the paper we adopt the convention that whenever we mention a word of length $t\ge 0$ where $t$ is not necessarily an integer, we actually mean a word of length $\lfloor t\rfloor$. 

We can now prove Theorem~\ref{thm:B} from the Introduction:

\setcounter{section}{1}
\setcounter{thm}{1}

\begin{thm}\label{t:main}
Let $N\ge 2$ and let $F_N=F(A)$ where $A=\{a_1,\dots, a_N\}$.  

Then there exist  constants $c(N)>0$, $D_1(N)>1$, $1>D_2(N)>0$, such that for $n\ge 1$ and for a freely reduced word $w_n\in F(A)$ of length $n$ chosen uniformly at random from the sphere $S(n)$ of radius $n$ in $F(A)$ we have
\[
1-P_{\mu_n}\left(\dsi(w_n)\ge c\log^{1/3} n\right)  = O\left((D_1)^{-n^{D_2}}\right)
\]
and 
\[
1-P_{\mu_n}\left(\dfi(w_n)\ge c\log^{1/5} n\right)  = O\left((D_1)^{-n^{D_2}}\right)
\]
so that
\[
\lim_{n\to\infty} P_{\mu_n}\left(\dsi(w_n)\ge c\log^{1/3} n\right) =1
\]
and 
\[
\lim_{n\to\infty} P_{\mu_n}\left(\dfi(w_n)\ge c\log^{1/5} n\right) =1
\]
\end{thm}

\begin{proof}

Choose $\epsilon>0$ and $0<\ell<1$ such that $\ell<1-2\epsilon$ (for concreteness we can take $\ell=1/2$ and $\epsilon=1/5$). 
Thus $1-\ell> \epsilon +(1-\ell)/2>0$.
Let $n_0\ge 1$ be such that for all $n\ge n_0$ we have \[\frac{2N-1}{4N}(0.99n)^{1-\ell}\ge (0.99n)^{\epsilon +(1-\ell)/2}\ge 1.\]  
(The choice of the number $0.99$ here is essentially arbitrary, and the argument would also work if $0.99$ is replaced by any other number sufficiently close to $1$.)
Let $C_1>1$ and $C_2>0$ be the constants provided by Corollary~\ref{c:CM}.  Note that we can assume that $0<C_2<1$ since decreasing $C_2$ preserves the validity of the conclusion of Corollary~\ref{c:CM}.

For $w_n\in S(n)$ denote by $w_n'$ the subword of $w_n$ obtained by removing the initial and terminal segments of length $0.005n$ from $w_n$. Then $|w_n'|=0.99n$ so that $w_n'\in S(0.99n)$. 
Since the uniform distribution on $A^{\pm 1}$ is stationary for the Markov chain $\mathcal X$, it follows that under the map $S(n)\to S(0.99n)$, $w_n\mapsto w_n'$ the uniform distribution $\mu_n$ on $S(n)$ projects to the uniform distribution $\mu_{0.99n}$ on $S(0.99n)$.

Let $H_n'$ be the event that for $w_n\in S(n)$ the word $w_n'$ satisfies the property that for every freely reduced word $\sigma\in F(A)$ with $|\sigma|=\ell\log_\lambda (0.99n)$ we have
\[\langle \sigma, w_n'\rangle \ge 1.\] 
Since for $n\ge n_0$ we have $\frac{2N-1}{4N}(0.99n)^{1-\ell}\ge (0.99n)^{\epsilon +(1-\ell)/2}\ge 1$, Corollary~\ref{c:CM} implies that
\[
1-P_{\mu_n}(H_n') =O((0.99n)^\ell C_1^{-(0.99n)^{C_2}} )= O\left( n^\ell (C_1)^{-0.99^{C_2}n^{C_2}} \right) =O\left((C_1')^{-n^{C_2'}}\right),
\]
where $C_1'=(C_1+1)/2$ and $C_2'=C_2/2$ (for the last inequality we use the fact that $0<C_2<1$). Note that $C_1'>1$ and $1>C_2'>0$.

Let $Q_n\subseteq S(n)$ be the event that for $w_n\in S(n)$ we have  $\iota(w_n)\le 0.001 n$. 
Lemma~\ref{lem:cr} implies that $P_{\mu_n}(Q_n)\ge 1-O(C_0^{-n})$ for some constant $C_0>1$.
Now let $H_n''$ be the set of all $w_n\in H_n'$ such that $\iota(w_n)\le 0.001 n$, that is, $H_n''=H_n'\cap Q_n$.

Then

\[
1-P_{\mu_n}(H_n'')=O\left((C_1')^{-n^{C_2'}}\right)-O(C_0^{-n})\ge_{n\to\infty} = O\left((D_1)^{-n^{D_2}}\right),
\]
where $D_1=\min\{C_0,C_1'\}$ and $D_2=\min\{C_2',1\}=C_2'$, so that $D_1>1$ and $1>D_2>0$.

We choose $c>0$ such that $c_0 c^3 \le \frac{\ell}{2\log(2N-1)}$, where $c_0>0$ is the constant provided by Proposition~\ref{p:v}.

Let $n\ge n_0$ and let $w_n\in S(n)$ be such that $w_n\in H_n''$.

Since $\iota(w_n)\le 0.001n$ and since $w_n'$ is the subword of $w_n$ obtained by removing the initial and terminal segments of length $0.005n$ from $w_n$, it follows that $w_n'$ is a subword of the cyclically reduced form $\tilde w_n$ of $w_n$. 

Let $d=\dsi(w_n)=\dsi(\tilde w_n)$.  We claim that $d \ge c\log^{1/3} n$.

Indeed, suppose not, that is, suppose that $d < c\log^{1/3} n$. Let $(\Gamma,x_0)$ be a $d$-fold cover of the $N$-rose $(R_N,\ast)$ such that $\tilde w_n$ lifts to a loop $\gamma_n$ from $x_0$ to $x_0$ in $\Gamma$ such that $\gamma_n$ belongs to a proper free factor of $\pi_1(\Gamma,x_0)$.  Note that since $\tilde w_n$ is cyclically reduced, the closed path $\gamma_n$ is also cyclically reduced.

Let $T$ be a maximal subtree of $\Gamma$ and let $v=v(\Gamma,T)$ be the freely reduced word in $F(A)$ with $|v|\le c_0d^3$ provided by Proposition~\ref{p:v}.  Thus $|v|\le c_0d^3\le c_0 c^3 \log n$.   

By definition of $H_n''$, the fact that $w_n\in H_n''$ implies that the word $w_n'$ contains as subwords all freely reduced words in $F(A)$ of length 
\[
\ell\log_\lambda (0.99n)=\frac{\ell}{\log(2N-1)} (\log n -|\log 0.99|)
\]
There is $n_1\ge n_0$ such that for all $n\ge n_1$ we have
\[
\frac{\ell}{\log(2N-1)} (\log n -|\log 0.99|)\ge \frac{\ell}{2\log(2N-1)} \log n.
\]
Hence for $n\ge n_1$ the word $w_n'$ contains as subwords all freely reduced words of length $\frac{\ell}{2\log(2N-1)} \log n$.
Since $|v|\le c_0 c^3 \log n\le \frac{\ell}{2\log(2N-1)} \log n$, it follows that $w_n'$ contains $v$ as a subword.

Recall that $w_n'$ is a subword of the cyclically reduced form $\tilde w_n$ of $w_n$. Therefore, by Proposition~\ref{p:v}, the path $\gamma_n$ in $\Gamma$, labeled by $\tilde w_n$, contains $\alpha(\Gamma,T)$ as a subpath.  Hence, by Corollary~\ref{cor:v}, $\gamma_n$ does not belong to a proper free factor of $\pi_1(\Gamma,x_0)$, yielding a contradiction. Thus $d=\dsi(w_n)\ge c\log^{1/3} n$, as claimed.

We have verified that for every $w_n\in H_n''$, where $n\ge n_1$, we have $\dsi(w_n)\ge c\log^{1/3} n$, and we also know that
\[
1-P_{\mu_n}(H_n'')=O\left((D_1)^{-n^{D_2}}\right).
\]
The conclusion of Theorem~\ref{t:main} regarding $\dsi(w_n)$ is established.

The proof of the conclusion of Theorem~\ref{t:main} regarding $\dfi(w_n)$ is identical, with Proposition~\ref{p:T} and Proposition~\ref{p:five} used instead of  Proposition~\ref{p:v} and Corollary~\ref{cor:v}. We leave the details to the reader.

\end{proof}

\setcounter{section}{8}

\section{Untangling closed geodesics on hyperbolic surfaces}\label{s:surfaces}

\subsection{Lower bounds for $\deg_{\Sigma,\rho}$ and $f_{\Sigma,\rho}$ for hyperbolic surfaces.}

We need the following well-known fact:
\begin{lem}\label{lem:simple}
Let $S$ be a compact connected surface with $b\ge 2$ boundary components such that $\pi_1(S)$ is free of rank $\ge 2$. Let $\gamma$ be an essential simple closed curve (possible peripheral) on $S$ and let $x\in S$ be a base-point for  $S$. Then the element of $\pi_1(S,x)$ given by any loop at $x$ corresponding to $\gamma$ belongs to a proper free factor of $\pi_1(S,x)$. 
\end{lem}
\begin{proof}

Without loss of generality we may assume that $x\in\gamma$.

By assumption, we have $\pi_1(S,x)=F_m$ with $m\ge 2$. Since $S$ has $b\ge 2$ boundary components, it follows that every boundary component (when realized as a loop at $x$) represents a primitive element of $F_m$.

Let $\gamma$ be an essential simple closed curve on $S$. If $\gamma$ is peripheral, then $\gamma$ is a primitive element of $F_m$ and thus belongs to a proper free factor of $F_m$.

Suppose now that $\gamma$ is non-peripheral.  Then cutting $S$ along $\gamma$ yields a nontrivial splitting of  $F_m=\pi_1(S)$ as an amalgamated product (if $\gamma$ is separating) or as an HNN-extension (if $\gamma$ is non-separating) over $\langle \gamma\rangle=\mathbb Z$. 
Suppose that $\gamma$ is separating, and it cuts $S$ into two compact surfaces $S_1$ and $S_2$ with $S_1\cap S_2=\gamma$ and $S_1\cup S_2=S$, each of $\pi_1(S_1), \pi_1(S_2)$ is free of rank $\ge 2$.  Thus $F_m=\pi_1(S,x)=\pi_1(S_1,x)\ast_{\gamma} \pi_1(S_2,x)$. The fact that $b\ge 2$ means that at least one of $S_1,S_2$ has $\ge 2$ boundary components. Assume for concreteness that $S_1$ has $\ge 2$ boundary components. Then $\gamma$ is primitive in $\pi_1(S_1,x)$. Thus we can find a free basis $a_1,\dots,a_m$ of $\pi_1(S_1,x)$ such that $m\ge 2$ and $\gamma=a_m$. Also choose a free basis $b_1,\dots, b_k$ of $\pi_1(S_2,x)$, where $k\ge 2$. Let $v\in F(b_1,\dots,b_k)=\pi_1(S_2,x)$ be the freely reduced word equal to $\gamma$ in $\pi_1(S_2,x)$.
Then the above splitting of $\pi_1(S,x)$ can be written as $\pi_1(S,x)=F(a_1,\dots,a_m)\ast_{a_m=v} F(b_1,\dots, b_k)$. By eliminating the generator $a_m$ from this presentation, we see that
$\pi_1(S,x)=F(a_1,\dots,a_{m-1},b_1,\dots, b_k)$. Thus $\gamma=v(b_1,\dots,b_k)$ belongs to a proper free factor $F(b_1,\dots,b_k)$ of $\pi_1(S,x)$, as required.
The case where $\gamma$ is non-separating is similar, and we leave the details to the reader.

Note that there is a general result (see, for example, \cite[Lemma 4.1]{BF} and \cite[Proposition 5.1]{Solie}) which says that whenever the free group $F_N$ (with $N\ge 2$) splits nontrivially as an amalgamated free product or an HNN-extension over a maximal infinite cyclic subgroup $\langle g\rangle$, then $g$ belongs to a proper free factor of $F_N$.
\end{proof}

The following proposition relates the degree function $\deg_{\Sigma,\rho}(\gamma)$ for curves in hyperbolic surfaces discussed in the Introduction, with the simplicity index $\dsi$ in free groups for curves contained in suitable subsurfaces: 

\begin{prop}\label{p:conn}
Let $(\Sigma,\rho)$ be a compact connected hyperbolic surface with (possibly empty) geodesic boundary. Let $\Sigma_1\subseteq \Sigma$ be a compact connected subsurface with $\ge 3$ boundary components, each of which is a geodesic in $(\Sigma,\rho)$.  Let $x\in \Sigma_1$ be a base-point. Then for every nontrivial element $g\in \pi_1(\Sigma_1,x)$ represented by a closed geodesic $\gamma_g$ on $\Sigma$ we have 
\[
\deg_{\Sigma,\rho}(\gamma_g)\ge \dsi(g; \pi_1(\Sigma_1, x)).
\]
\end{prop}
\begin{proof}
By assumption $\pi_1(\Sigma_1,x)\cong F_m$ is free of rank $m\ge 2$. 
The fact that $\Sigma_1$ is a subsurface of $\Sigma$ with geodesic boundary implies that if $g\in \pi_1(\Sigma_1,x)$ is a nontrivial element, then the shortest geodesic in $\Sigma$ in the free homotopy class of $g$ is contained in $\Sigma_1$.  Indeed, the universal cover $X:=\widetilde{(\Sigma_1, x)}$ is a convex $\pi_1(\Sigma_1,\ast)$-invariant subset of  $\widetilde{(\Sigma, x)}=\mathbb H^2$. Therefore for every nontrivial element $g\in  \pi_1(\Sigma_1,x)$ the axis $Axis(g)$ of $g$ in $\mathbb H^2$ is contained in $X$. The image of $Axis(g)$ in $\Sigma$ is the unique closed geodesic in the free homotopy class of $g$; the fact that $Axis(g)\subseteq X$ implies that this closed geodesic is contained in $\Sigma_1$, as claimed.

Now let $1\ne g\in \pi_1(\Sigma_1,x)$ and $\gamma_g$ be as in the assumptions of the proposition. Thus $\gamma_g$ is contained in $\Sigma_1$.

Let $d=\deg_{\Sigma,\rho}(\gamma_g)$. Let $p:\hat\Sigma\to\Sigma$ be a $d$-fold cover of $\Sigma$ such that $\gamma_g$ lifts to a simple closed geodesic $\hat\gamma_g$ in $\hat\Sigma$. Let $\hat\Sigma_1\subseteq \hat\Sigma$ be the connected component of the full preimage $p^{-1}(\Sigma_1)$ of $\Sigma_1$ containing $\hat\gamma_g$. Then $p:\hat\Sigma_1\to\Sigma_1$ is a $d'$-fold cover of $\Sigma_1$ with $d'\le d$. Pick a base-point $x'\in \hat\Sigma_1$ such that $p(x')=x$.

The cover $p:(\hat\Sigma_1,x')\to (\Sigma_1,x)$ corresponds to a subgroup $H\le \pi_1(\Sigma_1,x)$ of index $d'$, such that $p_\#(\pi_1(\hat\Sigma_1,x))=H$, and that $p_\#$ maps $\pi_1(\hat\Sigma_1,x')$ isomorphically to $H$.

Since $\hat\Sigma_1$ is a cover of $\Sigma_1$, the surface $\hat\Sigma_1$ has $\ge 2$ boundary components and $\pi_1(\hat\Sigma_1)$ is free of rank $\ge 2$.
By Lemma~\ref{lem:simple}, the fact that $\hat\gamma_g$ is an essential simple closed curve on $\hat\Sigma_1$ implies that $\hat\gamma_g$ corresponds an element $w\in \pi_1(\hat\Sigma_1,x')$ which belongs to a proper free factor of $\pi_1(\hat\Sigma_1,x')$.  Since $p(\hat\gamma_g)=\gamma_g$, we have $p_\#(w)=g\in H$. Since $p_\#$ maps  $\pi_1(\hat\Sigma_1,x')$ isomorphically to $H$, we conclude that $g$ belongs to a proper free factor of $H$.
Thus $H\le \pi_1(\Sigma_1,x)$, $[\pi_1(\Sigma_1,x):H]=d'$ and $g$ belongs to a proper free factor of $H$.  Therefore $d'\ge \dsi(g; \pi_1(\Sigma_1,x))$. 
Therefore
\[
\deg_{\Sigma,\rho}(\gamma_g)=d\ge d' \ge \dsi(g; \pi_1(\Sigma_1,x)),
\]
as required.
 
\end{proof}

\begin{thm}\label{t:main'}
Let $\Sigma$ be a compact  connected surface  with a hyperbolic structure $\rho$ and with (possibly empty) geodesic boundary.
Let $\Sigma_1\subseteq \Sigma$ be a compact connected subsurface with  $\ge 3$ boundary components, each of which is a geodesic in $(\Sigma,\rho)$. 
Let $x\in \Sigma_1$ and let $A$ be a free basis of $\pi_1(\Sigma_1,x)$.

Let $w_n\in F(A)=\pi_1(\Sigma_1,x)$ be a freely reduced word of length $n$ over $A^{\pm 1}$ generated by a simple non-backtracking random walk on $F(A)=\pi_1(\Sigma_1,x)$.
Let $\gamma_n$ be the closed geodesic on $(\Sigma,\rho)$ in the free homotopy class of $w_n$.

Then there exist  constants $c>0, K'\ge 1$ such that

\[
\lim_{n\to\infty} Pr( \deg_{\Sigma,\rho}(\gamma_n) \ge c \log^{1/3} n) =1
\]
and such that with probability tending to $1$ as $n\to\infty$
we have that $w_n\in \pi_1(\Sigma,x)$ is not a proper power and that $n/K'\le \ell_\rho(\gamma_n)\le K'n$.

\end{thm}
\begin{proof}
As we have seen in the proof of Proposition~\ref{p:conn}, the fact that $\Sigma_1$ is a subsurface of $\Sigma$ with geodesic boundary implies that if $g\in \pi_1(\Sigma_1,\ast)$ is a nontrivial element, then the shortest geodesic in $\Sigma$ in the free homotopy class of $g$ is contained in $\Sigma_1$.

By Theorem~\ref{t:main} and Lemma~\ref{lem:cr}, there exist an integer $n_0\ge 1$ such that  for $n\ge n_0$, with probability tending to $1$ as $n\to \infty$ we have that $w_n$ is not a proper power in $F(A)$, that $0.99n\le ||w_n||_A\le n=|w_n|_A$ and $\dsi(w_n; F(A))\ge c\log^{1/3} n$,  where $c=c(A)>0$ is the constant provided by Theorem~\ref{t:main} for the free group $F_m=F(A)$.

Proposition~\ref{p:conn} now implies that with probability tending to $1$ as $n\to\infty$ we have
\[
 \deg_{\Sigma,\rho}(\gamma_n) \ge \dsi(w_n; F(A))\ge c\log^{1/3} n.
\]

Finally, the fact that $\Sigma_1$ has geodesic boundary in $(\Sigma,\rho)$ also implies that there exists a constant $K\ge 1$ such that for every nontrivial element $g\in\pi_1(\Sigma_1,x)$ represented by a closed geodesic $\gamma$ on $(\Sigma,\rho)$ we have  $||g||_A/K\le \ell_\rho(\gamma)\le K||g||_A$.
Since with probability tending to $1$ as $n\to\infty$ we have $0.99n\le ||w_n||_A\le n=|w_n|_A$, it follows that for all sufficiently large $n$ with with probability tending to $1$ as $n\to\infty$ we have
$0.99 n/K \le \ell_\rho(\gamma_n)\le Kn$, as required.

\end{proof}

\begin{rk}
Theorem~\ref{t:main'} directly implies (e.g. by taking $\Sigma_1$ to be a suitable pair-of-pants subsurface) that if $(\Sigma,\rho)$ is a compact connected hyperbolic surface of genus $\ge 2$ with (possibly empty) geodesic boundary, then there exists $c'=c'(\Sigma)>0$ such that for every $L\ge sys(\rho)$ we have $f_\rho(L)\ge c' (\log L)^{1/3}$.
\end{rk}

\subsection{Lower bounds for $\deg_{\Sigma,\rho}^{fill}$ and $f_{\Sigma,\rho}^{fill}$ for hyperbolic surfaces.}

Our results about the behavior of $\dfi$ in free groups can also be used to obtain information about $\deg_{\Sigma,\rho}^{fill}$ for compact hyperbolic surfaces.

\begin{lem}\label{l:nonfill}
Let $(\Sigma,\rho)$ be a compact connected hyperbolic surface with $b\ge 1$ geodesic boundary components. Then the following hold:
\begin{enumerate}
\item If $\gamma$ is a non-filling closed geodesic on $(\Sigma,\rho)$ , then $\gamma$ represents a non-filling element of the free group $\pi_1(\Sigma)$.
\item For any closed geodesic $\gamma$ on $(\Sigma,\rho)$ we have $\deg_{\Sigma,\rho}^{fill}(\gamma)\ge \dfi(\gamma, \pi_1(\Sigma))$.
\end{enumerate}
\end{lem}
\begin{proof}
To see that (1) holds, let $\gamma$ be a non-filling closed geodesic on $(\Sigma,\rho)$. Then either $\gamma$ is contained in a proper compact connected subsurface $\Sigma_1$ of $(\Sigma,\rho)$ with geodesic boundary or $\Sigma-\gamma$ is a union of disks, peripheral annuli and non-peripheral annuli $A_1,\dots, A_k$ (where $k\ge 1$). In the latter case the simple closed geodesics $\alpha_1,\dots, \alpha_k$ homotopic to the core curves of $A_1,\dots, A_k$ are disjoint from $\gamma$, and we put $\Sigma_1$ to be the surface obtained by cutting $\Sigma$ open along the curves $\alpha_1,\dots, \alpha_k$.

In either case, cutting $\Sigma$ open along the boundary of $\Sigma_1$ provides a nontrivial graph-of-groups decomposition of  $\pi_1(\Sigma)$ with maximal cyclic edge groups and such that $\gamma$ belongs to a vertex group of this decomposition. Hence $\gamma$ is non-filling in $\pi_1(\Sigma)$. Thus (1) holds.

For (2), let $\gamma$ be a closed geodesic on $(\Sigma,\rho)$. Let $d=\deg_{\Sigma,\rho}^{fill}(\gamma)$ and let $\hat\Sigma\to\Sigma$ be a degree-$d$ cover such that $\gamma$ lifts to a closed non-filling geodesic $\hat\gamma$ on $\hat\Sigma$. This cover corresponds to a subgroup $H=\pi_1(\Sigma_1)\le \pi_1(\Sigma)$ of index $d$ containing the element $\gamma$.  The fact that $\hat\gamma$ is a non-filling  curve in $\Sigma_1$ implies, by part (1) of this lemma, that $\gamma$ is a non-filling element of $H=\pi_1(\Sigma_1)$. Therefore, by definition, $\dfi(\gamma, \pi_1(\Sigma))\le d=\deg_{\Sigma,\rho}^{fill}(\gamma)$, as required.
\end{proof}

\begin{thm}\label{t:sigmafill}
Let $(\Sigma,\rho)$ be a compact connected hyperbolic surface with $b\ge 1$ geodesic boundary components. Then there exists $C'>0$ such that for all sufficiently large $L$ we have
\[
f_{\Sigma,\rho}^{fill}(L)\ge C' \frac{\log L}{\log\log L}.
\]
\end{thm}
\begin{proof}
Let $\pi_1(\Sigma)=F_N=F(A)$ where $A=\{a_1,\dots, a_N\}$ with $N\ge 2$. The universal cover $X=(\tilde \Sigma,\tilde \rho)$ is a convex $\pi_1(\Sigma)$-invariant subset of $\mathbb H^2$. Therefore the orbit map $F(A)\to \mathbb H^2$, $w\mapsto w\ast$ (where $\ast\in \mathbb H^2$ is some basepoint) is a $\pi_1(\Sigma)$-equivariant quasi-isometry. Hence there exists $K\ge 1$ such that for every closed geodesic $\gamma$ on $(\Sigma,\rho)$ representing an element $w\in \pi_1(\Sigma)$ we have $||w||_A/K\le \ell_\rho(\gamma)\le K||w||_A$.

By Theorem~\ref{t:newmain} there exists a sequence of nontrivial cyclically reduced elements $w_n\in F(A)$ such that $||w_n||_A=n$ and that for all sufficiently large $n$ we have
\[
\dfi(w_n,F(A))\ge C \frac{\log n}{\log \log n},
\]
where $C>0$ is the constant provided by Theorem~\ref{t:newmain}. By Lemma~\ref{l:nonfill}, it follows that for all sufficiently large $n$ we have
\[
\deg_{\Sigma,\rho}^{fill}(\gamma) \ge \dfi(w_n,F(A))\ge C \frac{\log n}{\log \log n}.
\]
Since $||w||_A/K\le \ell_\rho(\gamma)\le K||w||_A$, the statement of the theorem now follows. 
\end{proof}

\begin{thm}
Let $(\Sigma,\rho)$ be a compact connected hyperbolic surface with $b\ge 1$ geodesic boundary components.  Let $A=\{a_1,\dots, a_N\}$ be a free basis of $\pi_1(\Sigma,x)$, so that $\pi_1(\Sigma)=F(A)$. 
Let $w_n\in F(A)=\pi_1(\Sigma,x)$ be a freely reduced word of length $n$ over $A^{\pm 1}$ generated by a simple non-backtracking random walk on $F(A)$.
Let $\gamma_n$ be the closed geodesic on $(\Sigma,\rho)$ in the free homotopy class of $w_n$.

Then there exist  constants $c_1>0, K_1\ge 1$ such that

\[
\lim_{n\to\infty} Pr( \deg_{\Sigma,\rho}^{fill}(\gamma_n) \ge c_1 \log^{1/5} n) =1
\]
and such that with probability tending to $1$ as $n\to\infty$
we have that $w_n\in \pi_1(\Sigma,x)$ is not a proper power and that $n/K_1\le \ell_\rho(\gamma_n)\le K_1 n$.

\end{thm}

\begin{proof}
The proof is essentially identical to the proof of Theorem~\ref{t:main'}, and we leave the details to the reader.
\end{proof}

\subsection{Degree and index functions based on the geometric intersection number}

Let $\Sigma$ be a compact connected surface admitting some hyperbolic structure (so that $\pi_1(\Sigma)$ is free of rank $\ge 2$). 
If $\rho$ is a hyperbolic metric on $\Sigma$ and $\gamma$ is a closed geodesic with respect to $\rho$ on $\Sigma$, we denote by $d_\rho(\gamma)$ the smallest degree of a finite cover of $\Sigma$ such that $\gamma$ lifts to a simple closed geodesic in that cover. 

We adopt the following conventions regarding the geometric intersection number for curves on surfaces. Let $S$ is a compact surface and $\alpha, \beta:\mathbb S^1\to S$ be homotopically nontrivial closed curves on $S$.
Then the geometric intersection number  $i([\alpha],[\beta])$ is defined as the minimum cardinality $|(\alpha_1\times \beta_1)^{-1}(\Delta)|$ where $\Delta \subseteq S\times S$ is the diagonal and  where $\alpha_1, \beta_1$ vary over all closed curves in the free homotopy classes $[\alpha], [\beta]$ respectively.  It is well-know that if $\rho$ is a hyperbolic structure on $S$ (where we always assume that the boundary curves of $S$, if any, are geodesic with respect to $\rho$) and if $\alpha,\beta$ are distinct closed primitive (i.e. not proper powers) geodesics on $S$ with respect to $\rho$ then $i([\alpha], [\beta])=|(\alpha\times \beta)^{-1}(\Delta)|$.  See \cite{CB} for a proof in the case of simple closed geodesics, and see p. 143 in \cite{Bon} and p. 99 in \cite{Bon86} for the general case.

Denote by $\mathcal C_\Sigma$ the set of free homotopy classes of essential closed curves on $\Sigma$ that are not proper powers in $\pi_1(\Sigma)$.  For $[\gamma]\in \mathcal C_\Sigma$ denote by $d_\Sigma([\gamma])$  the smallest degree of a finite cover of $\Sigma$ such that a representative of $[\gamma]$ lifts to a simple closed curve in that cover.  Note that if $\rho$ is a hyperbolic metric on $\Sigma$, then for every $[\gamma]\in \mathcal C_\Sigma$ there exists a unique closed $\rho$-geodesic $\gamma\in [\gamma]$ and $d_\rho(\gamma)=d_\Sigma([\gamma])$. Moreover, as noted above, in this case the geometric intersection number $i([\gamma],[\gamma])$ is realized by $\gamma$.

For an integer $m\ge 1$ we define $f_\Sigma(m)$ as the maximum of $d_\Sigma([\gamma])$ where $[\gamma]$ varies over all elements of $\mathcal C_\Sigma$ with $i([\gamma],[\gamma])\le m$. 
Similarly, for $[\gamma]\in \mathcal C_\Sigma$ denote by $d_\Sigma^{fill}([\gamma])$  the smallest degree of a finite cover of $\Sigma$ such that a representative of $[\gamma]$ lifts to a non-filling closed curve in that cover.  Then define $f_\Sigma^{fill}(m)$ as the maximum of $d_\Sigma^{fill}([\gamma])$ where $[\gamma]$ varies over all elements of $\mathcal C_\Sigma$ with $i([\gamma],[\gamma])\le m$.  Since simple curves are non-filling, we always have $d_\Sigma([\gamma])\ge d_\Sigma^{fill}([\gamma])$ and hence $f_\Sigma(m)\ge f_\Sigma^{fill}(m)$.

A result of Basmajian~\cite[Theorem 1.1]{Bas} (which also can be derived from the results of Bonahon~\cite{Bon}) states:
\begin{prop}\label{p:Basm}
Let $(\Sigma,\rho)$ be a connected compact hyperbolic surface with a (possibly empty) geodesic boundary. Then there exists a constant $K=K(\Sigma,\rho)\ge 1$ such that
for every closed geodesic $\gamma$ on $(\Sigma,\rho)$ we have
\[
i([\gamma],[\gamma])\le K \ell_\rho (\gamma)^2.
\]

\end{prop}

Theorem~\ref{t:sigmafill}  can be used to derive a lower bound for $f_\Sigma$:

\begin{thm}\label{t:intersect}
Let $\Sigma$ be a compact connected surface admitting some hyperbolic structure. 
Then there exist a constant $c=c(\Sigma)>0$ and an integer $m_0\ge 1$ such that for all $m\ge m_0$ we have 
\[
f_\Sigma(m)\ge f_\Sigma^{fill}(m)\ge c \frac{\log m}{\log\log m}.
\]
\end{thm}
\begin{proof}
Fix a hyperbolic metric $\rho$ on $\Sigma$.  By Proposition~\ref{p:Basm},  there exists a constant $K=K(\rho)>0$ such that for every $[\gamma]\in \mathcal C_\Sigma$ we have $i([\gamma],[\gamma])\le K \ell_\rho ([\gamma])^2$.
Let $C'=C'(\Sigma,\rho)>0$ be the constant provided by Theorem~\ref{t:sigmafill}. Then  Theorem~\ref{t:sigmafill} implies that there exist a sequence of closed geodesics $\gamma_n$ on $(\Sigma,\rho)$ and an integer $n_0\ge 1$ such that for every $n\ge n_0$ we have $ \ell_\rho(\gamma_n)\le n$ and $d_\Sigma^{fill}([\gamma_n])\ge C'\frac{\log n}{\log\log n}$.  Therefore $i(\gamma_n,\gamma_n)\le K \ell_\rho(\gamma_n)^2\le Kn^2$ for all $n\ge n_0$.  

Fix an integer $n_1\ge n_0$ such that for all integers $n\ge n_1$ we have $(n+1)^2\le 2n^2$.

Let $m\ge Kn_1^2$ be an integer. Choose an integer $n\ge n_1$ such that $Kn^2  \le m\le K(n+1)^2$.  Then \[  i([\gamma_n],[\gamma_n]) \le Kn^2 \le m\le  K(n+1)^2\le 2Kn^2\] and $n\ge \frac{\sqrt{m}}{\sqrt{2K}}$.

Therefore $i([\gamma_n],[\gamma_n])\le m$ and 
\[
d_\Sigma^{fill}([\gamma_n])\ge C' \frac{\log n}{\log\log n}\ge C' \frac{\log \frac{\sqrt{m}}{\sqrt{2K}}}{\log \log \frac{\sqrt{m}}{\sqrt{2K}}}=C'\frac{\frac{1}{2}\log m -\log \sqrt{2K} }{\log\left( \frac{1}{2}\log m -\log \sqrt{2K}\right)},
\]
and the statement of Theorem~\ref{t:intersect} follows.
\end{proof}

\begin{rk}

Note that the linear upper bound for $f_{\Sigma,\rho}(m)$, obtained by Patel~\cite{Patel} does not directly imply any upper bound for  $f_\Sigma(m)$. The reason is that  on a fixed hyperbolic surface there are arbitrarily long simple closed geodesics (which thus have self-intersection number $0$).  The lower bound for $f_\Sigma$ given by Theorem~\ref{t:intersect} was the first bound (upper or lower) known for $f_\Sigma$. Subsequent to our paper and in part motivated by it, Aougab, Gaster, Patel and Sapir~\cite{AGPS} proved that $f_\Sigma(m)=\Theta(m)$, that is $f_\Sigma(m)$ has precisely linear growth in $m$.

\end{rk}

\appendix

\section{Estimating the primitivity index function from below by the residual finiteness growth function function}\label{s:appendix}
\centerline{\emph{by Khalid Bou-Rabee}}
\centerline{City College of the City University of New York}
\bigskip

In this appendix we relate the primitivity index function $\fp(n;F_N)$ to the residual finiteness growth function introduced in \cite{BR1}.
Applying deep results of Gady Kozma and Andreas Thom \cite{KoTh} then improves the lower bounds for the primitivity index function to almost linear.

We first recall the residual finiteness growth function.
Let $G$ be a finitely generated, residually finite group. The \emph{divisibility function} $\D(g) = \D(g; G)$ is the minimum $[G:H]$ where $H$ varies over all subgroups of finite index in $G$ with $g \notin H$. For a fixed finite generating set $A \subset G$ the \emph{residual finiteness growth function} is
$\RF_{G,A}(n) := \max\{ \D(g; G) \;:\;g \in G, |g|_A \leq n, g \neq 1 \}.$ Here $|g|_A$ is the word-length of $g$ with respect to the word metric on $G$ corresponding to $A$.
In the case where $G$ is a nonabelian free group $F_N$ with word-length $|\cdot |_A$ given by a free basis $A$, we simply use this basis and denote the function by $\RF_G(n)$.

Next, we recall the primitivity index function introduced by Gupta and Kapovich above. 
Fix a free group $F_N$ of finite rank $N\ge 2$ with a free basis $A=\{a_1,\dots, a_N\}$.
The \emph{primitivity index} $\dpr(g) = \dpr(g; F_N)$ of an element $g \in F_N \setminus \{1 \}$ is the minimum $[F_N: H]$ where $H$ varies over all subgroups of finite index in $F_N$ containing $g$ as a primitive element.
Recall that the \emph{primitivity index function} is
\[
\fp(n;F_N)=\fp(n) := \max \{ \dpr(g) : g \in G, |g|_A \leq n, g \neq 1, g \text{  is not a proper power } \}.
\]

\begin{thm}\label{t:a1}
Let $G=F_N$ be a free group of finite rank $N\ge 2$.
Then $\RF_G(n) \leq \fp(4n+4)$ for all $n\ge 1$.
\end{thm}

\begin{proof}
For each $n\ge 1$ let $w_n$ be an element in $F_N$ with $|w_n|_A\le n$ such that $\D_G(w_n) = \RF_G(n)$.
In the free group $F_N$ commutativity is a transitive relation on the set of all nontrivial elements, and therefore there exists $a\in A$ such that $[w_n,a]\ne 1$.
Also, in a free group any two non-commuting elements freely generate a free subgroup of rank two. Thus $w_n,a$ freely generate a free subgroup of rank $2$ in $F_N$, and hence $\gamma_n := [w_n, w_n^a] \neq 1$. (In \cite{BRMR2,BRMR4} the property, that for every nontrivial $w\in F_N$ there exists $a\in A$ such that $[w,w^a]\ne 1$, is referred to as $F_N$ being \emph{1-malabelian}).
Note that $| [w_n, w_n^a]|_A\le 4n+4$. Since $\gamma_n$ is a nontrivial commutator in $F_N$, a result of Sch\"utzenberger~\cite{Sz} then implies that $\gamma_n$ is not a proper power in $F_N$.

Let $H$ be a finite-index subgroup of $G$ with $\gamma_n$ primitive in $H$.
If $w_n \in H$ and $w_n^a \in H$, then $[w_n, w_n^a] \in [H,H]$, and thus $[w_n, w_n^a]$ cannot be primitive in $H$.
Hence, $w_n$ or $w_n^a$ is not in $H$.
In either case, it follows that $[G:H] \geq D_G(w_n) =\RF_G(n)$.
Since $H$ was an arbitrary finite-index subgroup for which $[w_n, w_n^a]$ is primitive, it follows that
$\RF_G(n) \leq \fp(4n+4),$
as desired.
\end{proof}

A result of Kozma and Thom \cite{KoTh} about lower bounds for $\RF_{F_N}(n)$ now directly implies:
\begin{cor}\label{c:a2}
Let $G=F_N$ be free of finite rank $N\ge 2$. There exists a constant $C>0$ such that for all sufficiently large $n$ we have
\[
\fp(4n+4) \geq \exp \left( \left( \frac{\log(n)}{C \log\log(n)} \right)^{1/4} \right).
\]
If we assume Babai's Conjecture on the diameter of Cayley graphs of permutation groups, then for all sufficiently large $n$ we have
$
\fp(4n+4) \geq n^{\frac{1}{C\log\log(n)}}.
$
\end{cor}

At the time of this writing, for a nonabelian free group $G$, the best upper and lower bounds for $\fp(n)$ and $\RF_G(n)$ have the same asymptotic behavior. Is it true that $\fp(n)$ and $\RF_G(n)$ have the same asymptotic behavior?

\end{document}